\documentclass{amsart}
\usepackage{amsmath}
\usepackage{amsfonts}
\usepackage{amssymb}
\usepackage[usenames,dvipsnames]{xcolor}

\usepackage{hyperref}
\hypersetup{
    unicode=false,          		
    pdftoolbar=true,        		
    pdfmenubar=true,        	
    pdffitwindow=false,     		
    pdfstartview={FitH},    		
    linktocpage=true,			
    pdfnewwindow=true,      	
    colorlinks=true,       			
    linkcolor=blue,          		
    citecolor=PineGreen,    	
    filecolor=magenta,      		
    urlcolor=cyan           		
}

\theoremstyle{plain}
\newtheorem{theorem}{Theorem}[section]
\newtheorem{corollary}[theorem]{Corollary}
\newtheorem{lemma}[theorem]{Lemma}
\newtheorem{proposition}[theorem]{Proposition}

\theoremstyle{definition}
\newtheorem{definition}[theorem]{Definition}

\theoremstyle{remark}
\newtheorem{remark}[theorem]{Remark}

\numberwithin{equation}{section}	

\begin{document}

%
%
%

\newcommand{\refnote}[1]{}

%
%

\newcommand{\kbar}{{\overline{k}}}
\newcommand{\xbar}{{\overline{x}}}
\newcommand{\ubar}{{\overline{u}}}

\newcommand{\zetabar}{{\overline{\zeta}}}

\newcommand{\eps}{\varepsilon}

\newcommand{\bfs}{\mathbf{s}}
\newcommand{\bft}{\mathbf{t}}
\newcommand{\bfts}{{\mathbf{s}}}

\newcommand{\re}{\mathrm{Re}\,}

\newcommand{\frakq}{\mathfrak{q}}

\newcommand{\calB}{\mathcal{B}}
\newcommand{\calF}{\mathcal{F}}
\newcommand{\calS}{\mathcal{S}}
\newcommand{\calT}{\mathcal{T}}
\newcommand{\calQ}{\mathcal{Q}}
\newcommand{\calX}{\mathcal{X}^{\epsilon}}

\newcommand{\bbC}{\mathbb{C}}
\newcommand{\bbR}{\mathbb{R}}

\newcommand{\C}{\mathbb{C}}
\newcommand{\R}{\mathbb{R}}

\newcommand{\dotarg}{\, \cdot \,}
\newcommand{\dee}{\partial}
\newcommand{\dbar}{\overline{\partial}}

\newcommand{\darr}{\downarrow}
\newcommand{\rarr}{\rightarrow}
\newcommand{\dint}{\displaystyle{\int}}

\newcommand{\bigO}[1]{\mathcal{O}\left( {#1} \right)}
\newcommand{\norm}[1]{\left\| {#1} \right\|}

\title[Global Solutions for the Novikov-Veselov Equation]{Global Solutions for the Zero-Energy Novikov-Veselov Equation by Inverse Scattering}

\author{Michael Music}
\address{
Department of Mathematics, University of Michigan,
2074 East Hall,
530 Church Street,
Ann Arbor, Michigan 48109-1043}
\author{Peter Perry}
\address{Department of Mathematics, University of Kentucky, 
Lexington, Kentucky, 40506--0027}
\date{\today}

\begin{abstract}
Using the inverse scattering method, we construct global solutions to the Novikov-Veselov equation for real-­valued decaying initial data $q_0$ with the property that the associated Schr\"{o}dinger operator $-­\dbar_x \dee_x + q_0$ is nonnegative. Such initial data are either critical (an arbitrarily small perturbation of the potential makes the operator nonpositive) or subcritical (sufficiently small perturbations of the potential preserve non-negativity of the operator). Previously, Lassas, Mueller, Siltanen and Stahel proved global existence for critical potentials, also called potentials of “conductivity type.” We extend their results to include the much larger class of subcritical potentials. We show that the subcritical potentials form an open set and that the critical potentials form the nowhere dense boundary of this open set. Our analysis draws on previous work of the first author and on ideas of P. G. Grinevich and S. V. Manakov.
\end{abstract}
\maketitle

\section{Introduction}

The Novikov-Veselov (NV) equation is the completely integrable, nonlinear
dispersive equation%
\begin{eqnarray}
q_{t}  &  =	&		4\re\left(  4\partial^{3}q+\partial\left(  qw\right)
-E\partial w\right) \label{eq:NV_E}\\
\overline{\partial}w  &  =  &-3\partial q\nonumber
\end{eqnarray}
Here $E$ is a real parameter, the unknown function $q$ is a real-valued
function of two space variables and time, and the operators $\partial$ and
$\overline{\partial}$ are given by%
$$
\partial 					=	\frac{1}{2}\left(  \partial_{x_{1}}-i\partial_{x_{2}}\right) \quad
\overline{\partial}  	=	\frac{1}{2}\left(  \partial_{x_{1}}+i\partial_{x_{2}%
}\right)  .
$$
At zero energy ($E=0$) it can also be written (after trivial rescalings) as
\begin{equation} 
\label{eq:NV}
q_t = -\partial_z^3 q -\overline{\partial}^3q +3\partial_z(qu)+3\overline{\partial}_z(q\bar{u}), \quad \mbox{where} ~~ \overline{\partial}_z u=\partial_z q.
\end{equation}
The NV\ equation \eqref{eq:NV} generalizes the celebrated Korteweg-de Vries (KdV) equation
$$
q_{t}=-6qq_{x}-q_{xxx}
$$
in the sense that, if $q(x_{1},t)$ solves KdV and 
$u_{x_{1}}(x_{1},t)=-3q_{x_{1}}(x_{1},t)$, then $q\left(  x_{1},t\right)  $ solves NV.

Equation \eqref{eq:NV_E} was introduced by A.\ P.\ Veselov and S.\ P.\ Novikov \cite{VN:1984} because of its connection with isospectral flows for the Schr\"{o}dinger operator in two space dimensions at fixed energy $E$. In their original work, Veselov and Novikov considered periodic Schr\"{o}dinger operators and periodic solutions of \eqref{eq:NV_E};  P.\ G.\ Grinevich, S.\ V.\ Manakov, R.\ G.\ Novikov, and S.\ P.\ Novikov developed an  inverse scattering method for nonzero energy $E$ and initial data vanishing at infinity \cite{Grinevich:2000,GM:1988,GN:1987,GN:1988}. We refer the reader to Kazeykina \cite{Kazeykina:2012} for a review of results and results on long-time asymptotics of solutions to \eqref{eq:NV_E} at nonzero energy, and to \cite{CMMPSS:2014} for a comprehensive review of the literature and for a self-contained introduction to the inverse scattering method for the NV equation at zero energy. In what follows we will outline some of the salient features of the NV equation and discuss the motivations for our work.

\refnote{The next few paragraphs add background material on the Novikov-Veselov equation}
Although the Novikov-Veselov equation does not arise from a physical model, it defines a continuum 
of completely integrable models (indexed by the energy $E$) which are closely related to physically 
motivated, two-dimensional, dispersive nonlinear equations such as the celebrated Kadomtsev-Petviashvilli (KP) equations. Formal arguments (see Grinevich \cite{Grinevich:2000} and Klein-Mu\~{n}oz \cite[Appendix B]{KM:2016} for discussion) indicate that the KP I equation 
$$ u_t + 6 uu_x + u_{xxx} =  3 \dee_x^{-1} u_{yy} $$
can be realized as a scaling limit of 
\eqref{eq:NV_E} as $E \rarr +\infty$, and the KP II 
equation
$$ u_t + 6 uu_x + u_{xxx} = - 3 \dee_x^{-1} u_{yy} $$
can be realized as a scaling limit of \eqref{eq:NV_E} as $E \rarr -\infty$.  The KP equations 
model the evolution of nonlinear, long waves with slow dependence on the transverse coordinate.
Moreover, the NV equation exhibits many dynamical phenomena that also occur for the KP equations.  For example, the NV equation at nonzero energy has line soliton solutions and algebraic solitons (see \cite{KK:2017} for a numerical study of the stability of such structures, their relationship with analogous structures for the KP I and KP II equations,  and a review of the literature), while the NV equation at zero energy has a static lump solution. Thus, the zero-energy NV equation is part of a continuum of dispersive equations which exhibit important nonlinear wave phenomena and whose limits describe phenomena of physical interest. Significantly,  inverse scattering for the NV equation is determined by the (elliptic) stationary Schr\"{o}dinger equation, rather than the more challenging parabolic and time-dependent Schr\"{o}dinger equations that underlie the respective inverse theory for the KP I \cite{FA:1983,Manakov:1981}  and KP II \cite{ZS:1974,ABF:1983} equations.

A fundamental question for all dispersive PDEs is to characterize the Cauchy data for which solutions exist globally in time, and determine the long-time evolution of the solution. For completely integrable systems, one expects the large-time dynamics of solutions to be closely related to the spectral theory of the linear operator associated to the Cauchy data via completely integrability. Thus for example in the KdV equation, an isospectral flow for the one-dimensional Schr\"{o}dinger equation, initial data $q(\cdot,0)$ corresponding to a reflectionless potential generate pure soliton solutions, initial data with no bound states generates pure radiation, and initial data with bound states and nonzero scattering data generate solutions which 
resolve into solitons.

A similarly precise and comprehensive description of large-time dynamics for two-dimensional dispersive equations remains elusive; however, explicit solutions to NV at zero energy give some hints as to how the ``spectral theory of the Schr\"{o}dinger operator at zero energy'' should determine the asymptotics of solutions to \eqref{eq:NV}.
Taimanov and Tsarev \cite{TT:2008} constructed explicit solutions of \eqref{eq:NV} which blow up in finite time;
Kazeykina and Mu\~{n}oz \cite{KM:2018} constructed a one-parameter family $q_c$ of solutions where $c\in (-c_0,c_0)$
so that $q_0$ is a static (lump) solution, and for $\pm c \in (0,c_0)$ $q_c$ blows up in finite time $c$ but 
scatters to zero as $\pm t \rarr \infty$. These phenomena should (and can) be understood in terms of the spectral theory of the associated Schr\"{o}dinger operator: for all known examples of blow-up in finite time, the Schr\"{o}dinger operator associated to the initial data has a negative-energy bound state. One is naturally led to conjecture that, in the absence of bound states for the two-dimensional Schr\"{o}dinger operator, solutions of NV exist globally in time. Our results show that this is indeed the case, in a sense more precisely formulated below.

\refnote{The next paragraph briefly discusses the motivation for CGO solutions and their role in Schr\"{o}dinger scattering}
We will linearize and solve the Novikov-Veselov equation at zero energy using the scattering transform for the two-dimensional Schr\"{o}dinger equation at zero energy. This scattering transform is defined through unbounded solutions of the Schr\"{o}dinger equation introduced by Faddeev \cite{Faddeev:1965}. To understand the motivation for these solutions, it is easiest to consider the case of scattering at a fixed energy $E>0$. The usual `physical' solutions of Schr\"{o}dinger scattering theory, namely those solutions of
$(-\Delta + q) \psi = 0$ with
$$ 
\psi(x,k) \sim 
	e^{i(k_1x_1+k_2x_2)} + a(k,x/|x|) \frac{e^{i|k||x|}}{\sqrt{|k||x|}} + o\left(\frac{1}{\sqrt{|x|}}\right), \quad (k_1,k_2) \in \R^2, \, \, |k|^2 = E, 
$$
fail to give sufficient information to recover $q$. Indeed, at the level of the 
Born approximation, straightforward calculation shows that the scattering amplitude $a$ recovers the Fourier transform of $q$ at most inside a disc of radius $2\sqrt{E}$.  

Faddeev \cite{Faddeev:1965} proposed to study solutions with the leading asymptotics
$\exp i(k_1x_1+k_2x_2)$ for \emph{complex} $(k_1,k_2)\in \C^2$ with $k_1^2 + k_2^2 = E$ (which will \emph{grow} exponentially in some directions), and showed that this larger family of solutions produces scattering data sufficient to recover the potential. Faddeev's solutions are also termed CGO (for complex geometric optics) solutions. 

In the case $E=0$, the set of `physical' scattering solutions is actually trivial and one needs Faddeev's CGO solutions to define a scattering transform. In the zero-energy case, Faddeev's solutions are parameterized by $(k_1,k_2) \in \C^2$ satisfy $k_1^2+k_2^2=0$. There are two complex varieties satisfying this condition parameterized by $(k, \pm ik), k \in \C$, but symmetries of the Schr\"{o}dinger equation imply that the CGO solutions are dependent. Therefore, we will take $(k_1,k_2) = (k,ik)$ for $k \in \C$ in what follows. 
Thus $k_1 x_1 + k_2 x_2 = k(x_1+ix_2)$ and the phase may be written $kx$ where $x=x_1+ix_2$.
In what follows, we will treat $x=(x_1,x_2)\in\mathbb R^2$ as a complex number $(x_1+ix_2)$. 

Assuming $q\in L^p(\mathbb R^2)$ for 
a fixed $p\in(1,2)$, we say that $\psi(x,k)$ is a CGO solution for $k \in \bbC\setminus\{0\}$ if
\begin{equation*}
-\dee_x \dbar_x \psi + q(x)  \psi = 0, 
\quad
e^{-ikx}\psi(x,k) -1 
\in W^{1,\tilde{p}}(\bbR^2)
\end{equation*}
for some $p \in (1,2)$, where $\tilde{p}$ is the Sobolev conjugate:
\[\frac{1}{\tilde{p}} = \frac{1}{p}-\frac{1}{2}.\]
It is more convenient to work with the normalized solutions 
$$\mu(x,k)=\exp(-ikx) \psi(x,k)$$ 
which satisfy
\begin{equation}
\dbar_x(\partial_x+ik) \mu(x,k)=q(x)\mu(x,k)\\[0.2cm]
\label{mu.schro}
\end{equation}
where 
\begin{equation*}
\mu(\,\cdot\,,k)-1\in W^{1,\tilde p}(\bbR^2).
\end{equation*}
If, for some $k$,  equation \eqref{mu.schro} admits a solution $h(x)$ with $h(x) \in W^{1,\tilde p}(\bbR^2)$, uniqueness fails and $k$ is called an \emph{exceptional point}. Below we will describe a spectral condition on $q(x)$ which rules out exceptional points.

The associated scattering data $\bft$ for the potential $q$ is
\begin{equation}
\label{t}
\bft(k) = \int_{\bbR^2} e_k(x) q(x)\mu(x,k) \, dm(x)
\end{equation}
where $dm(\dotarg)$ denotes Lebesgue measure on $\bbR^2$ and
\[ e_k(x) = \exp i\left(kx+\kbar \xbar\right). \]
The map $\calT:q \rarr \bft$ defined by \eqref{mu.schro} and \eqref{t} is the \emph{direct scattering map}.

The normalized CGO solution $\mu(x,k)$ can be recovered through the scattering data, and the potential can in turn be recovered from the reconstructed solutions $\mu(x,k)$.  Indeed, if
$$ \bfs(k) = \frac{\bft(k)}{\pi \kbar}, $$
the function $\mu(x,k)$ satisfies the $\dbar$-equation 
\begin{equation}
\label{eq:dbar-k.static}
 \left(\dbar_k \mu \right)(x,k) 
 	= e_{-x}(k) \bfts(k)\,  \overline{\mu(x,k)}, \quad
\mu(x,\dotarg)-1 \in L^{\tilde{r}}(\bbC)
\end{equation}
where $r \in (p,2)$ and $\tilde{r}$ is the Sobolev conjugate of $r$.
The potential $q(x)$ is recovered in turn using the reconstruction formula
\begin{equation}
\label{eq:q.recon}
q(x) = 
	\frac{i}{\pi}\dbar_x 
		\left( 
				\int_{\bbC} e_{-x}(k) \bfts(k) \, \overline{\mu(x,k)} \, dm(k) 
		\right).
\end{equation}
The map $\calQ: \bft \rarr q$ defined by \eqref{eq:dbar-k.static} and \eqref{eq:q.recon} is the \emph{inverse scattering map}.

Formal computations \cite{BLMP:1987} (see also \cite[\S 3.6]{CMMPSS:2014}) show that, if $q(x,\tau)$ solves NV, then the function $\bft(k,\tau)=\calT(q(\dotarg,\tau))(k)$ should evolve by the linear evolution
\begin{equation}
\label{t.flow}
\bft(k,\tau)=\exp \left( i \tau \left(k^3 + \kbar^3\right) \right) \bft(k,0). 
\end{equation}

Putting all of these ingredients together--the direct scattering transform $\calT$, the linearization \eqref{t.flow}, and the inverse transform $\calQ$--we get the solution formula
\begin{equation}
\label{pre.ISM}
	 q(x,\tau) = \calQ \left[ \exp(i((\dotarg)^3+\overline{(\dotarg)^3} \tau) \left( \calT q_0 \right)(\dotarg) \right](x).
\end{equation}
From \eqref{eq:dbar-k.static} and \eqref{t.flow}, it follows that, to solve the NV equation, we need to solve the $\dbar$-problem with parameters
\begin{equation}
\label{eq:dbar-k}
{\displaystyle \left(\dbar_k \mu \right)(x,k,\tau) = e_{-x}(k) \bfts(k,\tau)\,  \overline{\mu(x,k,\tau)}} \quad
\mu(x,\dotarg,\tau)-1 \in L^r(\bbC)
\end{equation}
and recover
\begin{equation}
\label{ISM}
q(x,\tau) = \frac{i}{\pi}\dbar_x \left(\int_{\bbC} e_{-x}(k) \bfts(k,\tau) \overline{\mu(x,k,\tau)} \, dm(k)\right)
\end{equation}
where we  define 
\[ \bfts(k,\tau)=\frac{\bft(k,\tau)}{\pi \kbar}.\]

As we will show, global existence for the formal solution \eqref{pre.ISM} within a class of real-valued, regular, and decaying initial data is guaranteed by a simple spectral condition on the initial data. For a real-valued function $q \in L_{\mathrm{loc}}^p(\bbR^2)$ for some $p>1$, we say that the operator 
$L=-\bar\partial_x\partial_x+q$ is nonnegative if, for all $\varphi \in C_c^\infty(\bbR^2)$,
\begin{equation}
\label{q.ps}
\int_{\mathbb R^2} \left| (\dee_x \varphi)(x) \right|^2 + q(x) \left| \varphi(x) \right|^2 \, dm(x)\geq 0.
\end{equation}

Murata \cite{Murata:1986} defined a trichotomy of potentials for the zero-energy Schr\"{o}dinger operators 
(see Simon \cite{Simon:1981} for an earlier and closely related trichotomy in the study of Schr\"{o}dinger semigroups). This trichotomy plays a crucial role in our analysis and may distinguish distinct dynamical regimes for the NV equation.

\begin{definition} 
Fix $p>1$. A potential $q\in L^p_{\mathrm{loc}}(\mathbb R^2)$ is
\begin{itemize}
\item[(i)] \emph{subcritical} if there is a positive Green's function for the operator $L$,
\item[(ii)]\emph{critical} if there is no positive Green's function but $L\geq 0$, and
\item[(iii)] \emph{supercritical} if $-\dbar_x\partial_x+q$ is not nonnegative.
\end{itemize}
\label{def:murata}
\end{definition}
 Murata (see \cite[Definition 2.1]{Murata:1986} and accompanying discussion) shows these are the only three possibilities, so the condition \eqref{q.ps} implies that $q$ is either critical or subcritical. 
Theorem 2.12 of \cite{Schro-Ops} implies that if $q$ is uniformly locally in $L^p(\mathbb R^2)$ for  $p>1$, then $-\Delta+q\geq 0$ if and only if there exists a positive distributional solution to the Schr\"odinger equation (we use here the fact that, for $p>1$, the uniformly locally $L^p$ potentials are contained in Kato class $K_2$ considered in \cite{Schro-Ops}; see \cite[\S 1.2, Remark 1(b)]{Schro-Ops} and see \cite[\S 2.5]{Schro-Ops} for further discussion and references to the original work of Allegretto, Moss, and Piepenbrink). 
 Additionally, if $q(x)$ has sufficient decay, the equation $-\dbar_x\dee_x \psi + q\psi=0$  admits a positive, \emph{bounded}, distributional solution $\psi$  in the critical case, and a positive solution with \emph{logarithmic growth} in the subcritical case \cite[Theorem 5.6]{Murata:1986}. These positive solutions are unique up to constant multiples. In Section \ref{sec:topo} we show that subcritical potentials are an open set in the topology $L^p_\rho(\mathbb R^2)$ with $p\in(1,2)$ and $\rho>1+2/\tilde p$ and the critical potentials are the boundary of this set.

The asymptotic behavior as $|x| \rarr \infty$ of the positive solution is closely related, via a Fourier-like duality, to the asymptotic behavior as $|k| \darr 0$ of the scattering transform $\bft(k)$. Indeed, Nachman \cite[Theorem 3.3]{Nachman:1996} showed that, for critical potentials, $| \bft(k)  | \leq C|k|^\eps$ for some $\eps>0$ and small $k$. This means that $\bfts(k)=\bigO{|k|^{\eps-1}}$ is in $L^p_{\mathrm{loc}}(\bbC)$ for any $p \in (1,\infty)$ and the integral operator \eqref{eq:Txtau} for the problem \eqref{eq:dbar-k} has range in $L^\infty$ functions using \eqref{eq:P-2} of Lemma \ref{lemma:P1}. By contrast, Music \cite{Music:2013} showed that, for subcritical potentials, 
\[ \bfts(k) \underset{|k| \rarr 0}{\sim} -\frac{1}{\kbar \log(|k|^2)}
		\left( 1+ \bigO{\left(-\log(|k|^2\right)^{-1}}\right)\]
and the same key integral operator does not have range in bounded functions. This singularity is the cause of some technical pain in our computations but, thanks to a density argument (see Lemma \ref{lem:mapcont} and Corollary \ref{cor:invsymmetry}), does not pose essential problems. 
In Proposition \ref{prop:critorsubcrit} we prove that any scattering transform that has this singularity and also has certain decay properties (see Definition \ref{def:xnp}) must come from a critical or subcritial potential.

To state the theorem, denote by $L_\rho^p(\bbR^2)$ the space
\[ \left\{ f \in L^p(\bbR^2): \left(1+|\dotarg|^2\right)^{\rho/2} f(\dotarg) \in L^p(\bbR^2) \right\} \]
and by $W^{n,p}_\rho(\bbR^2)$ the space of measurable functions with $n$ weak derivatives in $L^p_\rho(\bbR^2)$. 

\begin{theorem}
Let $q_0\in W^{5,p}_\rho(\mathbb R^2)$ for $p\in(1,2)$ and $\rho>1$ be a critical or subcritical potential. Then $q(x,\tau)$ given by \eqref{pre.ISM} is a global-in-time, classical solution of NV \eqref{eq:NV} with  $\lim_{\tau \rightarrow 0} q(x,\tau)=q_0(x)$ pointwise.
\label{thm:nv}
\end{theorem}

The regularity assumption guarantees that the inverse scattering method will yield a classical solution of NV. By contrast, Angelopoulos \cite{Angelopoulos:2013} recently proved local well-posedness for NV in the space $H^s(\bbR^2)$ for any $s>\frac{1}{2}$. The inverse scattering method captures global properties of the solution at the expense of more stringent regularity and decay assumptions on the initial data.

The condition that the potential is subcritical or critical is necessary in the above theorem using our methods. This is because in \cite{MPS:2013}, Music, Perry, and Siltanen construct radial, compactly supported, supercritical potentials $q_0$ where the scattering transform $\bft(k)$ has a circle of singularities and the current formulation of the inverse scattering method fails.  P.G. Grinevich and R.G. Novikov \cite{GN:2012} give explicit examples of point potentials with similar contour-type singularities.
Taimanov and Tsarev (see \cite{TT:2013} and references therein) construct examples of supercritical potentials that give rise to solutions of the zero-energy NV equation that blow up in finite time.

The technical core of our result is a careful analysis of the $\dbar$-problem 
\eqref{eq:dbar-k} with parameters $x, \tau$. 
The first technical challenge is that the scattering data $\bft(k)$ for our class of potentials may be mildly singular as $|k| \darr 0$ (see Lemma \ref{lemma.scattering}). We deal with the singularity by proving continuity properties of the direct and inverse scattering maps and approximating $\bft(k)$, the scattering transform of a  potential $q$ in our class by a smooth function with compact support away from $k=0$. 

The second technical challenge is that, in order to prove that the reconstructed potential
\eqref{ISM} solves the NV equation, we must first prove that it is a real-valued function. We will show that this is the case by adapting ideas of Grinevich \cite{Grinevich:2000} and Grinevich-Novikov \cite{GN:1988}. 

Lassas, Mueller, Siltanen, and Stahel \cite{LMSS:2011} study the same problem for initial data of conductivity type which have very well-behaved scattering transforms as shown by Nachman in his 1996 paper \cite{Nachman:1996} on Calderon's problem.  These authors obtain a number smoothness and decay results for this class, but they cannot prove that the inverse scattering method yields solutions to the Novikov-Veselov equation. Combining our result with the earlier one from Lassas, Mueller, and Siltanen \cite{LMS:2007} gives global solutions to Novikov-Veselov which satisfy the decay condition $|q(x,\tau)|\leq \langle x\rangle^{-2}$ if we start with a critical potential $q(x,0)\in C^\infty_c(\mathbb R^2)$. Although this special case has already been proved by Miura map methods by Perry \cite{Perry:2014}, the proof give here is arguably more direct and simpler.

In Section \ref{sec:pre}, we state some theorems that will be used throughout the paper. In the first half, Section \ref{sec:pre}, we recall facts about the $\dbar$ problem, and in Section \ref{sec.music}, we recall results of Music \cite{Music:2013} on the properties of the scattering transform for critical and subcritical potentials  \cite{Music:2013}. Music's characterization of the range of $\calT$ will allow us to analyze the inverse map $\calQ$ rigorously. In Section \ref{sec:topo}, we prove that the set of subcritical potentials is open and that the set of critical potentials is its nowhere dense boundary. In Section \ref{sec.continuous}, we establish continuity in $q$ for the direct scattering map $\calT$, and continuity in $\bfts(k)$ of the inverse scattering map $\calQ$.  In Section \ref{sec.symmetry}, we exploit ideas of P.G. Grinevich  and S.V. Manakov \cite{GM:1988} to prove necessary and sufficient conditions for the scattering data $\bft(k)$ to be the scattering data of a real potential $q(x)$ repairing a gap \cite{Siltanen:2014} in the proof of \cite[Theorem 4.1]{LMSS:2011}. We apply the ideas of Grinevich-Manakov \cite{GM:1988} directly to the zero-energy case instead of treating it as a limit of negative-energy inverse scattering (see also Grinevich \cite{Grinevich:2000}). With this result, we prove that the reconstructed potential $q(x,\tau)$ stays real for all time. Crucially, the reality of the inverse map is needed to show that the evolved $\mu(x,k,\tau)$ continue to solve the problem \eqref{mu.schro}. In Section \ref{sec.solving}, we follow the method outlined in the review article of Croke, Mueller, Music, Perry, Siltanen and Stahel \cite{CMMPSS:2014} to show that the potential solves the Novikov-Veselov equation. We discuss some open problems in section \ref{sec.open}.


\subsection*{Acknowledgment}
We are grateful to the referee for a careful reading of this paper, for helpful suggestions,  and for pointing out an error in a previous version of the manuscript.

\section{Preliminaries}

\subsection{Notation}
\label{sec:note}
In what follows, $\norm{u}_{p}$ denotes the $L^p(\bbR^2)$ or $L^p(\bbC)$ norm, while $\norm{u}_{p'}$ denotes the $L^{p'}$ norm where $p^{-1} + (p')^{-1} = 1$.  The space $W^{n,p}(\bbR^2)$ consists of $L^p$ functions with weak  derivatives up to order $n$ in $L^p(\bbR^2)$, while the space $L^p_\rho(\bbR^2)$ consists of complex-valued measurable functions on $\bbR^2$ with $(1+|\dotarg|^2)^{\rho/2} f(\dotarg) \in L^p(\bbR^2)$. The space $W^{n,p}_\rho(\bbR^2)$ consists of $L^p_\rho(\bbR^2)$ functions with weak derivatives up to order $n$ belonging to $L^p_\rho(\bbR^2)$. Finally if $X$ is a Banach space, we denote by $C(\bbC,X)$ the space of continuous $X$-valued functions carrying the norm $\norm{f}_{C(\bbC,X)} = \sup_{x \in X} \norm{f(x)}_{X}$, and $\calB(X)$ denotes the Banach space of bounded operators from $X$ to itself with the usual norm.

\subsection{The $\dbar$-Problem}
\label{sec:pre}
Here we review some important facts about the $\dbar$ operator and its inverse that will be used in what follows. Further details and references can be found, for example, in \cite[section 2]{Perry:2013}
or \cite[chapter 4]{AIM:2009}.
We denote by $P$ the integral operator
\[ \left[ P f\right](z) = \frac{1}{\pi} \int_{\bbC} \frac{1}{z-\zeta} f(\zeta) \, dm(\zeta)\]
where, in the integrand,  $z=z_1+iz_2$ and $\zeta=\zeta_1+i\zeta_2$. We recall without proof the following basic estimates. In what follows, $C_p$ (resp. $C_{p,q}$) denotes a numerical constant depending only on $p$ (resp. $p,q$).

\begin{lemma}
\label{lemma:P1}$~$(i) For any $p\in\left(  2,\infty\right)  $ and $f\in
L^{2p/(p+2)}({\bbC})$,
\begin{equation}
\left\Vert Pf\right\Vert _{p}\leq C_{p}\left\Vert f\right\Vert _{2p/(p+2)}.
\label{eq:P-1}
\end{equation}
Moreover, $\nabla Pf \in L^{2p/(p+2)}(\bbC)$ and 
\begin{equation}
\label{eq:P-1a}
\left\Vert \nabla Pf \right\Vert_{2p/(p+2)} \leq C_p \left\Vert f \right\Vert_{2p/(p+2)}.
\end{equation}
(ii) For any $p,q$ with $1<q<2<p<\infty$ and any $f\in L^{p}\left(
{\bbC}\right)  \cap L^{q}\left(  {\bbC}\right)  $, the
estimate%
\begin{equation}
\left\Vert Pf\right\Vert _{\infty}\leq C_{p,q}\left\Vert f\right\Vert
_{L^{p}\cap L^{q}} \label{eq:P-2}%
\end{equation}
holds. 
\newline(iii) If $v\in L^{s}({\bbC})$, $p\in (2,\infty)$ and $q\in(2,\infty)$ with $q^{-1}%
+1/2=p^{-1}+s^{-1}$, then for any $f\in L^{p}({\bbC})$,
\begin{equation}
\left\Vert P\left(  vf\right)  \right\Vert _{q}\leq C_{p,q}\left\Vert
v\right\Vert _{s}\left\Vert f\right\Vert _{p}. \label{eq:P-4}%
\end{equation}
\end{lemma}

The operator $P$ is the solution operator for the problem $\dbar u = f$:

\begin{lemma}
\label{lemma:P2}Suppose that $p\in(2,\infty)$, that $u\in L^{p}(\bbC)$,
that $f\in L^{2p/(p+2)}({\bbC})$, and that $\overline
{\partial}u=f$ in distribution sense. Then $u=Pf$. Conversely, if $f\in
L^{2p/(p+2)}({\bbC})$ and $u=Pf$, then $\dbar u=f$ in
distribution sense.
\end{lemma}

We will need the following generalized Liouville Theorem for quasi-analytic functions, due in this form to Brown and Uhlmann. Music \cite{Music:2013} has an extension of this lemma to include certain negatively weighted $L^p$ spaces.

\begin{theorem}[{\cite[Corollary 3.11]{BU:1997}}] 
\label{thm:vanish}
Suppose that $u \in L^p(\bbC) \cap L^2_{\mathrm{loc}}(\bbC)$ for some $p \in [1,\infty)$ and that
\[ \dbar u = a u + b \ubar \]
for $a$ and $b$ belonging to $L^2(\bbC)$. Then $u\equiv 0$.
\end{theorem}

Finally, solutions of $\dbar u = f$ for rapidly decaying $f$ have a large-$z$ expansion.

\begin{lemma}
\label{lemma:P3}Suppose that $p\in\left(  2,\infty\right)  $, that $u\in
L^{p}({\bbC})$, 
that $f\in L^{2p/(p+2)}_n({\bbC})$, 
and that $\dbar u=f$. Then%
\[
z^{n}\left[  u(k)-\sum_{j=0}^{n-1}\frac{1}{z^{j+1}}\int\zeta^{j}%
f(\zeta)~dm(\zeta)\right]  \in L^{p}({\bbC}).
\]
\end{lemma}

We will also use the following estimates from \cite{BKS:2003} on the fundamental solution kernel for the linear problem
$ v_\tau = \dee_x^3 v + \dbar_x^3 v $
associated to the NV equation.

\begin{lemma}
\label{lemma:Itau}
Let
\begin{equation}
\label{Itau}
I_\tau(x) = \int e^{i\tau(k^3+\bar k^3)-i(kx+\bar k\bar x)} \, dm(k).
\end{equation}
Then:
\[I_\tau(x) = \tau^{-2/3}I_1\left(x\tau^{-1/3}\right)\]
and the estimates 
\begin{eqnarray*}
\left| I_1(x) \right| &\leq &C\left(1+|x|\right)^{-1/2},\\
\left| \nabla_x I_1(x) \right| &\leq & C
\end{eqnarray*}
hold.
\end{lemma}

For a proof see \cite[Proposition 5.4]{BKS:2003}.

\subsection{Properties of the scattering transform}
\label{sec.music}

The following is a reformulation of Theorem 1.4 and Lemma 5.2 from Music \cite{Music:2013}. Instead of stating the lemma in terms of the space that the potential $q$ lives, we state the assumptions that the scattering transform satisfies. This is the purpose of defining the space $\calX_{n,r}$ in Definition \ref{def:xnp}. See Remark \ref{rem:Xnp} for more details.

 For subcritical potentials, we let $\psi(x)$ be the unique positive solution of $-\dbar_x\dee_x \psi + q \psi = 0$ such that 
\begin{equation}
\psi(x) = a\log|x| + O(1).
\label{eq:a}
\end{equation}
We also set
\begin{equation}
c_\infty = \lim_{R\to\infty} \frac{1}{\pi R^2} \int_{|x|<R} \left( \psi(x)-a\log|x| \right) \,dm(x).
\label{eq:cinfty}
\end{equation}
Note that for critical potentials, we have $a=0$ and $c_\infty =1$, and the following results are originally due to Nachman \cite[Section 3]{Nachman:1996} in this case.
\begin{lemma}
\label{lemma.scattering}
Let $q\in L^p_\rho(\mathbb R^2)$ with $p\in(1,2)$, and $\rho>1$, and suppose that $q$ is either critical or subcritical. Let $\psi(x)$ denote a positive function solving $(-\dbar_x\partial_x+q)\psi=0$. Then $q(x)$ has no exceptional points, and the scattering transform $\bft(k)$ of $q(x)$ has the following properties:
\begin{enumerate}
\item[(i)] For small $k$ and some $\eps>0$, we have
\[\bft(k)= \frac{1}{2}\frac{\pi a}{c_\infty-a(\log|k|+\gamma)}+O(|k|^\eps),\]
where $c_\infty$, $a$ are defined in equations \eqref{eq:a} and \eqref{eq:cinfty}.
\item[(ii)] $\bft(k)/\kbar \in L^r(|k|>\epsilon)$ for $r \in (\tilde p',\infty)$ and every $\epsilon>0$. Moreoever, $\bft(k)/\kbar \in C(\mathbb C\setminus{0})$.
\item[(iii)] If $q\in W^{n,p}_\rho(\mathbb R^2)$ for $n\geq 1$ then $|k|^n{\bfts}(k) \in L^r(\mathbb C)$.
\end{enumerate}
\label{lem:props-of-transform}
\end{lemma}

In fact, we will use the above properties for scattering transforms coming from real potentials to define a class of scattering transforms.

\begin{definition}
\label{def:xnp}
The space $\calX_{n,r}$ for $n\geq 1$, $r\in (1,2)$, and $\epsilon>0$ is the closure of $C^\infty_c(\mathbb C)$ functions which satisfy the relation $\bar k f(k) = -k\overline{f(-k)}$ in the norm
\[\|f\|_{\calX_{n,r}} = \|f\|_{L^2(\mathbb C)}+\|(\,\cdot\,)^nf(\,\cdot\,)\|_{L^{r'+\epsilon}\cap L^r(\mathbb C)}+\|(\,\cdot\,)^nf(\,\cdot\,)\|_{L^{r}(\mathbb C)}.\]
\end{definition} 
\begin{remark}
\label{rem:Xnp}
The symmetry requirement guarantees that the reconstructed potential is real-valued; see Lemmas \ref{lem:symmetry} and \ref{lem:invsymmetry} of Section \ref{sec.symmetry}. The conditions in this definition are preserved under the linearized NV flow 
\eqref{t.flow}. By Lemma \ref{lem:props-of-transform} scattering transforms $\bfts(k)$ coming from potentials $q\in W^{n,p}_\rho(\mathbb R^2)$ are in $\calX_{n,r}$ for all $r\in(\tilde p',\infty)$ and $\epsilon>0$.

These properties of $\bfts(k)\in \calX_{n,r}$ are used in Music \cite{Music:2013} in the inverse problem to reconstruct $\mu(x,k)$ and obtain a large-$k$ expansion of $\mu(x,k)$. For the large-$k$ expansion to exist pointwise, we need the decay $k^n \bfts(k)\in L^{2+\epsilon}\cap L^{2-\epsilon}(\mathbb C)$ (see Lemmas 5.5 and 5.6 of \cite{Music:2013}). We need this decay for all $k$ away from the origin, and in this region it suffices to have $\bfts \in L^2(\bbC)$ instead of $L^{2+\epsilon}(\bbC)$. This is why we only need the homogeneous norm on $L^{r'+\epsilon}(\bbC)$. With this norm and the embedding $\calX_{n+1,r}\subset L^1_n$, we have $\mu(x,~\cdot~)-1\in L^{\tilde r}(\mathbb C)$ and for every $x$ we get the expansion \eqref{eq:largek-mu}.
\end{remark}

In what follows, we define the operator $T$  by
\begin{equation}
\label{eq:Txtau}
\left(T_{x,\tau} f\right)(k)
	= P_k \left[\bfs(\dotarg,\tau) e_{-x}(\dotarg)\overline{f} \right](k)
\end{equation}
where 
\[ \left[ P_k f\right](k) = \frac{1}{\pi} \int_{\bbC} \frac{1}{k-\zeta} f(\zeta) \, dm(\zeta).\]

We now recall from \cite{Music:2013} that the solution of \eqref{eq:dbar-k} is differentiable in the parameters $x$ and $\tau$.

\begin{lemma}
\label{lem:mu-diff}
If  $\bfs\in \calX_{n,r}$ then the unique solution $\mu(x,\, \cdot \, ,\tau)-1\in L^{\tilde{r}}(\mathbb C)$ of equation \eqref{eq:dbar-k} is $\alpha$  times differentiable in $x$ and $m$ times differentiable in $t$ for $n\geq 3m+|\alpha|$. Additionally, the derivatives of the map $(x,\tau)\to\mu(x,\, \cdot \,,\tau)-1\in L^{\tilde{r}}(\mathbb C)$ satisfy $\partial_\tau^m D^{\alpha}_{x}\mu(x,\, \cdot \,,\tau)\in L^{\tilde{r}}(\mathbb C)$. The derivatives are given by
\begin{equation}
\label{mu.deriv}
\partial_\tau^m D^{\alpha}_{x} \mu(x,k,\tau)=[I-T_{x,\tau}]^{-1}P_k
	f(x,\dotarg,\tau)
\end{equation}
where
\begin{equation}
\label{f}
	f(x,k,\tau)
			=\left[ \partial_\tau^m D^{\alpha}_{x},\bfts(k,\tau)e_{-x}(k) \right]
					\overline{\mu(x,k,\tau)},
\end{equation}
and $P_k[{\bfts}(k,\tau)e_{-x}(k)f(x,k,\tau)]  \in L^{\tilde{r}}(\mathbb C)$.
\end{lemma}

Finally, it follows from Lemma \ref{lemma:P3} that $\mu(x,k,\tau)$ has a large-$k$ expansion. We refer the reader to \cite{Music:2013} for a full proof.

\begin{lemma}
If $\bfs\in \calX_{n,r}$ and if $\mu$ solves the equation \eqref{eq:dbar-k} with $\mu(x,\cdot)-1 \in L^{\tilde{r}}(\mathbb C)$, then $\mu$ admits the large-$k$ expansion
\begin{equation}
\mu(x,k,\tau)= 1+\sum_{j=1}^{n}\frac{a_j(x,\tau)}{k^{j}}+ o\left(|k|^{-n}\right)
\label{eq:largek-mu}
\end{equation}
for fixed $x$ and $\tau$.
Moreover, we may take $\alpha$ spatial derivatives and $m$ time derivatives with $n\geq |\alpha|+3m$ to get 
\[\partial_\tau^m D^{\alpha}_{x}(\mu(x,k,\tau)-1)=\sum_{j=1}^{n-|\alpha|-3m}\frac{\partial_\tau^m D^{\alpha}_{x} a_j(x,\tau)}{k^{j}}+o\left(|k|^{-n+|\alpha|+3m}\right).\]
\end{lemma}

We will use the above lemma to take up to four spatial derivatives, or one time and one space derivative of $a_1(x,\tau)$. This requires $\bft(k)\in \calX_{5,r}$ and is why we assume that $q(x,0)\in W^{5,p}_\rho(\mathbb R^2)$. It is also important to notice that these are not weak derivatives as $-i\dbar^{-1}q = a_1(x,\tau)$ and $W^{5,p}_\rho(\mathbb R^2) \subset C^3(\mathbb R^2)$.

\section{Topology of the set of subcritical potentials}
\label{sec:topo}

The set of subcritical potentials is actually bigger than the set of critical potentials. Murata \cite[Theorem 2.4]{Murata:1986} shows that the set of critical potentials is not stable under compact perturbations. Here we go a step further and show the subcritical potentials are an open set in the topology of $L^{p}_\rho(\mathbb R^2)$ for certain $p$ and $\rho$. In the corollary to the theorem we show that subcritical potentials are open in the $W^{1,p}_\rho(\mathbb R^2)$ topology for the same $p$ and $\rho$ as in Theorem \ref{thm:nv}. 
\begin{theorem}
The set of subcritical potentials is  open in the $L^{p_1}_{\rho_1}(\mathbb R^2)$ topology for $p_1\in (1,2)$ and $\rho_1>1+2/\tilde p_1$ and the set of critical potentials is its boundary.
\end{theorem}

\begin{proof}
We prove that the set of subcritical potentials is open by finding positive logarithmically growing solutions, $\phi_v$, to $(-\Delta+v)\phi_v = 0$ for all $v$ close enough to a given subcritical potential $q$.

Let $G_0(x) = -1/(2\pi)\log|x|$. Then,
by Nachman \cite[Lemma 3.2]{Nachman:1996} we have for $f\in L^{p_2}_{\rho_2}(\mathbb R^2)$ with $\rho_2>1$ and $p_2\in(1,2)$
\[ \left\|G_0\ast f + \frac{1}{2\pi} \log|x|\int f \, dm(x)\right\|_{L^{\tilde p_2}}\leq c\|f\|_{L^{p_2}_{\rho_2}}\]
and
\[\|\nabla G_0 \ast f \|_{L^{\widetilde{p_2}}} \leq c\|f\|_{L^{p_2}}.\]
Modifying both of these slightly, we let 
\[Tf = G_0\ast f +\frac{1}{2\pi} \log(|x|+e) \int f dm(x) \]
 and get
\[ \|Tf\|_{W^{1,\tilde p_2}}\leq c\|f\|_{L^{p_2}_{\rho_2}}.\]
In particular, by Sobolev embedding we have 
\[\|Tf\|_{L^\infty} \leq \|f\|_{L^{p_2}_{\rho_2}}.\]

Let $\phi_q$ be a positive solution to the Schrodinger equation and let $c_\infty$ be given by \eqref{eq:cinfty}. By Nachman \cite[Lemma 3.5]{Nachman:1996} for critical potentials and Music \cite[Theorem 1.4]{Music:2013} for subcritical potentials with $c_\infty\neq 0$, we have that $[I+G_0\ast (q\,\cdot\,)]$ is invertible for critical/subcritical $q$ on $W^{1,\tilde p}_{-\beta}(\mathbb R^2)$ for $\beta>2/\tilde p$. If $c_\infty = 0$, we may use the same argument from Music  \cite[Theorem 1.4 and Lemma 4.1]{Music:2013} to scale $q$, and therefore $\phi_q$, so that $c_\infty \neq 0$ and we may reduce to the case. So for some $\epsilon>0$ the operator $[I+G_0\ast (v\,\cdot\,)]$ is invertible for all potentials $v$ satisfying $\|q-v\|_{L^p_\rho}<\epsilon$. Thus we have $W^{1,\tilde p}_{-\beta}$ distributional solutions to $(-\Delta+v)\phi_v = 0$ by taking $\phi_v = [I+G_0\ast (v\,\cdot\,)]^{-1}c_\infty$. Because $q$ is subcritical $\phi_q = a\log|x|+O(1)$ is a positive solution. The solutions $\phi_v$ satisfy the  equation
\[\phi_v = c_\infty - G_0\ast (v \phi_v).\] 

Taking the difference between these two solutions we find
\begin{eqnarray*}
\phi_q-\phi_{v} &=& G_0\ast (v\phi_v-q\phi_q) + \frac{1}{2\pi}\log(|x|+e)\int (v\phi_v-q\phi_q)dm(x)\\
&\qquad& -\frac{1}{2\pi}\log(|x|+e)\int (v\phi_v-q\phi_q)dm(x)\\
&=& T(v\phi_v-q\phi_q)-\frac{1}{2\pi}\log(|x|+e)\int (v\phi_v-q\phi_q)dm(x)
\end{eqnarray*}
which becomes
\[ |\phi_q-\phi_v| \leq c\|v\phi_v-q\phi_q\|_{L^{p_2}_{\rho_2}} +\frac{1}{2\pi}\log(|x|+e)\|v\phi_v-q\phi_q\|_{L^{1}}\]
for $p_2\in (1,2)$ and $\rho_2>1$. We may choose these constants by taking $\beta$ close enough to $2/\tilde p$ so that $\rho_2 = \rho_1-\beta>1$.
Also, we have the embedding $\|f\|_{L^1}\leq c\|f\|_{L^{p_2}_{\rho_2}}$ so
\[\phi_v\geq \phi_q-c\|v\phi_v-q\phi_q\|_{L^{p_2}_{\rho_2}} -\frac{1}{2\pi}\log(|x|+e)\|v\phi_v-q\phi_q\|_{L^{p_2}_{\rho_2}}.\]
The function $\phi_q$ is in $W^{1,\tilde p}_{-\beta}(\mathbb R^2)\subset C(\mathbb R^2)$ with logarithmic growth at infinity, so $\phi_q>c'>0$. With this strict positivity and growth of $\phi_q$, the right hand side is positive and logarithmically growing for $\|v\phi_v-q\phi_q\|_{L^{p_2}_{\rho_2}}$ small. This norm is small because $\|\phi_v-\phi_q\|_{W^{1,\tilde p}_{-\beta}}\leq c\|q-v\|_{L^p_\rho}$ by the second resolvent formula applied to the operators $[I+G_0\ast(v\,\cdot\,)]$. The function $\phi_v$ is a positive distributional solution to $(-\Delta+v)\phi_v=0$ so by \cite[Theorem 2.12]{Schro-Ops} $v$ is critical or subcritical.  It follows from Murata \cite[Theorem 5.6]{Murata:1986} that $v$ is subcritical because the positive solutions for critical potentials in this weighted space have the asymptotics $\phi= c+o(1)$ whereas the positive solutions for subcritical potentials obey $\phi= a\log|x|+O(1)$.

In Theorem 2.4 of the same paper, Murata showed that nonnegative compact perturbations of critical potentials are subcritical and nonpositive perturbations are supercritical. Thus critical potentials are a subset of the boundary of subcritical potentials. Looking at the form \eqref{q.ps} we see that limits of subcritical potentials must either be subcritical or critical. Therefore, critical potentials form the entire boundary. 
\end{proof}

Now we apply the above lemma to the set of potentials under consideration in Theorem \ref{thm:nv}.

\begin{corollary}
The set of subcritical potentials is open in the $W^{1,p}_\rho(\mathbb R^2)$ topology for $p\in (1,2)$ and $\rho>1$ and the set of critical potentials is its boundary.
\end{corollary}
\begin{proof}
Be Sobolev embedding $W^{1,p}_\rho(\mathbb R^2) \subset L^s_\rho(\mathbb R^2)$ for all $s\in (p,\tilde p)$. In particular we may choose $s<2$ but close enough to $2$ so that $\rho>1+2/\tilde s$. The result then follows from the previous theorem.
\end{proof}

\section{Continuity of Maps}
\label{sec.continuous}

In this section, we prove that the scattering data, $\bft$, depends continuously on the potential $q$ and that the solution of \eqref{eq:dbar-k} is continuous in $\bfts(k)$ in a suitable sense.  Using the continuity of the inverse scattering map defined by the $\dbar$-problem, we can show that our reconstructed $q(x,\tau)$ is differentiable. In fact by approximating the scattering transform $\bfs(k)$ by  $\bfs_\epsilon(k)\in C^{\infty}_c(\mathbb C)$ in the $\calX_{n,r}$ norm, we approximate the solution $q(x,\tau)$ by the reconstructed $q_\epsilon(x,\tau)$. In this way, the results showing that each $q_\epsilon(x,\tau)$ solves the Novikov-Veselov equation carry over to the general subcritical case so long as the initial data $q_0$ has enough derivatives.

\subsection{Continuous dependence of $\bft$ on $q$}
Recall from Nachman \cite[\S 1]{Nachman:1996}  that if $p \in (1,2)$, $u \in W^{1,\tilde{p}}(\bbR^2)$, $f \in L^p(\bbR^2)$, $k \in \bbC \setminus \{ 0 \}$, and
$\dbar_x(\dee_x+ik)u=f$ in distribution sense, then
$$ u = g_k*f $$
for a convolution kernel $g_k$ which satisfies the estimate

$$ \norm{g_k*f}_{W^{1,\tilde{p}}} \leq c_0(k,p) \norm{f}_p.$$

Let us define
$$ S_k(q) f  = g_k*(qf). $$
It follows from the estimates above that for $q \in L^p(\bbR^2)$, 
\begin{equation}
\label{Sk.norm}
\norm{S_k(q)}_{\calB(W^{1,\tilde{p}})} \leq c_0(k,p) \norm{q}_p
\end{equation}
while $S_k(q) 1 := g_k*q$ satisfies
\begin{equation}
\label{Sk.1.norm}
\norm{S_k(q) 1}_{W^{1,\tilde p}} \leq c(k,p) \norm{q}_p.
\end{equation}
In \eqref{Sk.norm} we used the Sobolev inequality $\norm{f}_{\infty} \leq c_p \norm{f}_{W^{1,\tilde{p}}}$.

We now consider the continuity of 
$ \mu(x,k;q)$ defined by
$$ \mu(x,k;q) -1 = (I-S_k(q))^{-1} S_k(q) 1$$
as a function of $q$. 

\begin{lemma}
\label{lemma:mu.q}
Fix $p \in (1,2)$, $k \in \bbC \setminus \{ 0 \}$, and $q_0 \in L^p(\bbR^2)$. 
Suppose that
 $$R(q_0):=(I-S_k(q_0))^{-1}$$
exists as a bounded operator from $W^{1,\tilde{p}}(\bbR^2)$ to itself. 
There is a number $r>0$ and a constant $c(p,k,q_0)$ so that for all $q \in L^p(\bbR^2)$ with $\norm{q_0-q}_p \leq r$, the estimate
$$
\norm{\mu(\dotarg,k,q) - \mu(\dotarg,k,q_0)}_{W^{1,\tilde{p}}} \\
	\leq c(p,k,q_0)\norm{q-q_0}_p
$$
holds.
\end{lemma}

\begin{proof}
In what follows $c$ denotes a constant depending only on $k$, $p$, and $q_0$ whose value may vary from line to line. We will use the estimates
\begin{eqnarray}
\label{Sk.delta}
\norm{S_k(q)-S_k(q_0)}_{\calB(W^{1,\tilde{p}})} 
	&\leq& 	c \norm{q-q_0}_p\\
\label{Sk.1.delta}
\norm{S_k(q) 1 - S_k(q_0) 1}_{W^{1,\tilde{p}}}
	&\leq& c \norm{q-q_0}_p
\end{eqnarray}
which follow from \eqref{Sk.norm}, \eqref{Sk.1.norm}, and linearity in $q$.

By the second resolvent formula, for $q$ sufficiently close to $q_0$, we have
$$ 
R(q) - R(q_0) = R(q_0) \left( S_k(q)-S_k(q_0) \right) R(q) 
$$
so that
$$ 
R(q) = 		\left[
						I-R(q_0)\left(S_k(q)-S_k(q_0)\right)
				\right]^{-1} 
				R(q_0) 
$$
These computations can be justified if 
$ \norm{R(q_0)\left[S_k(q)-S_k(q_0)\right]}_{\calB(W^{1,\tilde{p}})} <1/2$. From estimate \eqref{Sk.delta}, we see that it suffices to take
$$ \norm{q-q_0}_p \leq \left(2 c \norm{R(q_0)}_{\calB(W^{1,\tilde p})} \right)^{-1}. $$
Thus, taking $r = \left(2 c \norm{R(q_0)}_{\calB(W^{1,\tilde p})} \right)^{-1}$, we have 
\begin{equation}
\label{res.nbhd.bd}
\norm{R(q)}_{\calB(W^{1,\tilde{p}})} \leq 2 \norm{R(q_0)}_{\calB(W^{1,\tilde p})}.
\end{equation}
 Using the identity
$$ 
\mu(x,k,q) - \mu(x,k,q_0) = 
		R(q) S_k(q) 1 - 
		R(q_0) S_k(q_0) 1,
$$
the estimates \eqref{Sk.delta}, \eqref{Sk.1.delta}, \eqref{res.nbhd.bd}, and the second resolvent formula, we recover the claimed estimate.
\end{proof}

Now we can prove continuity of $\bft$ as a function of $q$. 

\begin{lemma}
\label{lemma:t.cont}
Suppose that $p \in (1,2)$, $\rho>1$, that $q \in L^p_\rho(\bbR^2)$, and $\{ q_n \}$ is a sequence from 
$L^p_\rho(\bbR^2)$ with $q_n \rarr q$ in $L^p_\rho(\bbR^2)$. Finally, let $\bft_n = \calT (q_n)$ and $\bft = \calT(q)$. Then, for all non-exceptional nonzero $k$, $\bft_n(k) \rarr \bft(k)$ pointwise.
\end{lemma}
\begin{proof}
Using Lemma \ref{lemma:mu.q} we estimate
\begin{eqnarray*}
|\bft_n(k) - \bft(k)| &\leq& \left|\int e_k(x)(q_n(x)-q(x))\mu_n(x)\,dm(x)\right|\\
&&\quad +\left|\int e_k(x)q(x))(\mu_n(x)-\mu(x))\,dm(x)\right|\\
&&\leq \|q_n-q\|_{L^1} \|\mu_n\|_{L^\infty} +\|q\|_{L^1}\|\mu_n-\mu\|_{L^\infty}\\[0.2cm]
&&\leq \|q_n-q\|_{L^p_\rho} \left(\|\mu_n-1\|_{W^{1,\tilde p}}+1\right) +\|q\|_{L^p_\rho}\|\mu_n-\mu\|_{W^{1,\tilde p}}
\end{eqnarray*}
and conclude that $\bft_n(k) \rarr \bft(k)$. 
\end{proof}

\subsection{Continuous dependence of reconstructed $q$ on $\bft$}

For notational convenience,  we combine the time dependence of the scattering data 
$$\bfts(k,\tau)=\exp(i\tau(k^3+\bar k^3))\bfts(k,0)$$ with the spatial oscillation $e_{-x}(k)$ into the factor $e^{i\tau S}$ where
\[S(x,k,\tau) = -\frac{kx+\bar k \bar x}{\tau}+(k^3+\bar k^3).\]

We note the following estimates. 
\begin{lemma}
\label{lemma:Txtau}
Suppose that $\bfs \in L^2(\bbC)$ and $\tilde{r}>2$. Then $(I-T_{x,\tau})^{-1}$ exists as an operator in $\calB(L^{\tilde{r}})$ and
$$ 
\sup_{(x,\tau) \in \bbR^2 \times \bbR} \norm{(I-T_{x,\tau})^{-1}}_{\calB(L^{\tilde{r}})}
< \infty.
$$
\end{lemma}

\begin{proof}
We will prove the assertion in four steps. First, we will show that $T_{x,\tau}$ is a compact operator on $L^{\tilde{r}}(\bbC)$ with 
\begin{equation}
\label{OpNormEst}
\sup_{(x,\tau) \in \bbR^2 \times \bbR} \norm{T_{x,\tau}}_{\calB(L^{\tilde{r}})} \leq C(\tilde{r}) \norm{\bfs}_{L^2(\bbC)}.
\end{equation}

Second, we will prove that $\ker(I-T_{x,\tau})$ is trivial, so that, by the Fredholm alternative,
the operator $\left( I - T_{x,\tau} \right)^{-1}$ exists for all 
$(x,\tau) \in \bbR^2 \times \bbR$.

Third, we will show that
\begin{equation}
\label{LP1}
\lim_{T \rarr \infty} 
	\sup_{x \in \bbR^2, ~ |\tau| \geq T} 
		\norm{T_{x,\tau}^2}_{\calB(L^{\tilde{r}})} = 0
\end{equation}
and that for each $T>0$, 
\begin{equation}
\label{LP2}
\lim_{R \rarr \infty} 
	\sup_{|\tau| \leq T, ~ |x| \geq R} 
		\norm{T_{x,\tau}^2}_{\calB(L^{\tilde{r}})} = 0.
\end{equation}
It then follows that there are numbers $R$ and $T$ for which $\norm{T_{x,\tau}^2}_{\calB(L^{\tilde{r}})} < 1/2$ whenever $|x|\geq R$ or $|\tau|\geq T$, and  that therefore
\begin{equation}
\label{eq:lessthan2}
\norm{\left( I - T_{x,\tau}^2 \right)^{-1}}_{\calB(L^{\tilde{r}})} \leq 2
\end{equation}
on this set. Using the identity $(I-A)^{-1} = (I-A^2)^{-1}(I+A)$, the estimates \eqref{OpNormEst}, and \eqref{eq:lessthan2}, we recover
$$
\norm{\left( I - T_{x,\tau} \right)^{-1}}_{\calB(L^{\tilde{r}})} 
\leq 
2\left(1+C\norm{\bfs}_{L^2(\bbC)}\right)
$$
for all such $(x,\tau)$. 

Fourth, we will use a continuity-compactness argument to show that for any positive numbers $R$ and $T$,
$$
\sup_{|x| \leq R, \, |\tau| \leq T} \norm{\left( I - T_{x,\tau} \right)^{-1}}_{\calB(L^{\tilde{r}})}
$$
is finite.

(1) The norm estimate is immediate from \eqref{eq:P-4} with $p=q=\tilde{r}$. To show that $T_{x,\tau}$ is compact, it follows from \eqref{OpNormEst} and density that we may take $\bfs \in C_0^\infty(\bbC)$ without loss. It suffices to prove that the Banach space adjoint  $T_{x,\tau}'=e^{i\tau S}\bfs P_k$ is a compact operator on $L^{\tilde{r}'}(\bbC)$. Let $\Omega$ be the support of $\bfs$. If $f \in L^{\tilde{r}'}(\bbC)$ then $P_k f \in L^{2\tilde{r}/(\tilde{r}-2)}(\bbC)$ by \eqref{eq:P-1} while $\nabla P_k f \in L^{\tilde{r}'}(\bbC)$ by \eqref{eq:P-1a}. Thus 
$$
\norm{T_{x,\tau}'f}_{W^{1,\tilde{r}'}} \leq C \left(1+|\Omega|^{1/2} \right) \norm{f}_{{\tilde{r}'}}
$$ 
and compactness follows from the Rellich-Kondrachov Theorem.

(2) To see that $\ker\left(I-T_{x,\tau}\right)$ is trivial, note that by Lemma \ref{lemma:P2}, $\psi \in \ker\left(I-T_{x,\tau}\right)$ if and only if $\dbar \psi = e^{i\tau S}\bfs \psi$. Since $\bfs \in L^2(\bbC)$ and $\psi \in L^{\tilde{r}}(\bbC)$, it follows from Theorem \ref{thm:vanish} that $\psi =0$.

(3) We will prove \eqref{LP1} and \eqref{LP2} by proving the corresponding estimates for $(T_{x,\tau}')^2$ on $L^{\tilde{r}'}(\bbC)$. First, note that $T_{x,\tau}' = e^{i\tau S} W$ where $W=\bfs P_k$ is a compact operator independent of $(x,\tau)$. We have the estimate $\norm{T_{x,\tau}'}_{\calB(L^{\tilde{r}'})} \leq \norm{W}_{\calB(L^{\tilde{r}'})}$ uniform in $(x,\tau)$.  For any $\epsilon>0$ there is a finite-rank operator $F$ on $L^{\tilde{r}'}(\bbC)$   so that $\norm{W-F}_{\calB(L^{\tilde{r}'})} < \epsilon$. The operator  $F$ is a finite sum $\sum_j \langle \varphi_j, \dotarg \rangle \psi_j$ where 
$\varphi_j \in L^{\tilde{r}}(\bbC)$, $\psi_j \in L^{\tilde{r}'}(\bbC)$, and $\langle \dotarg,\dotarg \rangle$ is the usual dual pairing of $L^{\tilde{r}}(\bbC)$ and $L^{\tilde{r}'}(\bbC)$. A density argument shows that we may take $\psi_j$ and $\varphi_j$ in $\calS(\bbC)$ without loss.  Thus, it suffices to show that 
\begin{equation}
\label{LP1.pre}
\lim_{T \rarr \infty}  \sup_{x \in \bbR^2, ~ |\tau| \geq T} \left| \langle \varphi, e^{i\tau S} \psi \rangle \right| = 0
\end{equation}
and that, for each fixed $T>0$,
\begin{equation}
\label{LP2.pre}
\lim_{R \rarr \infty} 
	\sup_{|\tau| \leq T, ~ |x| \geq R} \left| \langle \varphi, e^{i\tau S} \psi \rangle \right| = 0
\end{equation}
where $\varphi$ and $\psi$ belong to $\calS(\bbC)$. 

Let $f(k)=\psi(k)\varphi(k)$. A short computation shows that, for $S=S(x,k,\tau)$,
$$
\langle \varphi, e^{i\tau S} \psi \rangle
=
\int I_\tau(x-y) \left(\calF^{-1}f\right)(y) \, dy
$$
where $I_\tau$ is given by \eqref{Itau} and
$$
\left( \calF^{-1} f\right)(x) = \frac{1}{\pi}\int e_{-k}(x) f(k)\, dm(k).
$$
We now appeal to Lemma \ref{lemma:Itau} to estimate
\begin{equation}
\label{IPest.1}
\left| \langle \varphi, e^{i\tau S} \psi \rangle \right|
\leq
\int  \tau^{-2/3} (1+|(x-y)/\tau^{1/3}|)^{-1/2}  \left| \left(\calF^{-1}f\right)(y) \right|  \, dm(y).
\end{equation}
To prove \eqref{LP1.pre}, we estimate the right-hand side of \eqref{IPest.1} by $T^{-2/3}\norm{\calF^{-1}f}_{L^1}$. This gives \eqref{LP1}.  To prove \eqref{LP2.pre} we make the change of variables $\xi =(x-y)/\tau^{1/3}$ and obtain
\begin{eqnarray*}
\left| \langle \varphi, e^{i\tau S} \psi \rangle \right|
		&\leq&	\int (1+|\xi|)^{-1/2}  
					\left| \left(\calF^{-1}f\right)(x-\xi \tau^{1/3}) \right|  
				\, dm(\xi)  \\
\nonumber
		&\leq& C(f,N) (I_1(x,\tau)+I_2(x,\tau))				
\end{eqnarray*}
where $C(f,N)$ depends on the Schwartz seminorms of $f$ and
\[I_1(x,\tau) = \int_{|\xi|\leq |x|/(2T^{1/3})} (1+|\xi|)^{-1/2}(1+|x-\xi\tau^{1/3}|)^{-N} \, dm(\xi)\]
\[I_2(x,\tau) = \int_{|\xi|\geq |x|/(2T^{1/3})} (1+|\xi|)^{-1/2}(1+|x-\xi\tau^{1/3}|)^{-N} \, dm(\xi).\]

For $|\tau|>1$, equation \eqref{LP2.pre} now follows from the preceding arguments and the elementary estimates
\begin{eqnarray*}
\left| I_1(x,\tau) \right| &\leq& c(N)\left( \frac{|x|}{2T^{1/3}} \right)^2 (1+|x|)^{-N} \\
\left| I_2(x,\tau) \right| &\leq& c(N)\left(1+\frac{|x|}{2T^{1/3}}\right)^{-1/2}.
\end{eqnarray*}
For $|\tau|<1$, we note that the we can rewrite the operator $T_{x,\tau}'$ as
\[T_{x,\tau}' = e^{i(\tau-3)S(x,k,\tau-3)} \left[e^{3i(k^3+\bar k^3)} \bfs(k) P_k\right].\]
Equation \eqref{LP2.pre} then holds for $|\tau-3|>1$ as well. This implies \eqref{LP2}.

(4) We claim that the map
\begin{eqnarray*}
\bbR^2 \times \bbR	&\longrightarrow& \calB(L^{\tilde{r}}(\bbC))\\
(x,\tau)	&\mapsto	&\left(I-T_{x,\tau}\right)^{-1}
\end{eqnarray*}
is continuous. If so, it is continuous on the compact set $\{(x,\tau): |x| \leq R, |\tau| \leq T \}$ hence bounded there. By the second resolvent formula and duality, it suffices to show that the map
\begin{eqnarray*}
\bbR^2 \times \bbR	&\longrightarrow& \calB(L^{\tilde{r}'}(\bbC))\\
(x,\tau)	&\mapsto&  T_{x,\tau}'
\end{eqnarray*}
is continuous. But, viewed as multiplication operators on 
$L^{\tilde{r}'}(\bbC)$,  $e^{i\tau' S(\dotarg,x',\tau')} \rarr e^{i\tau S(\dotarg,x,\tau)}$ as $(x',\tau') \rarr (x,\tau)$ 
in the strong operator topology. Since 
$$
T_{x',\tau'}' - T_{x,\tau}' = \left( e^{i\tau S(\dotarg,x',\tau')}-e^{i\tau S(\dotarg,x,\tau)} \right) W
$$ for a fixed compact operator $W$, it follows that 
$$
\norm{T_{x',\tau'}'-T_{x,\tau}'}_{\calB(L^{\tilde{r}'})} \rarr 0, \quad
 (x',\tau') \rarr (x,\tau).
$$
This proves the 
required continuity.
\end{proof}

Given the uniform resolvent bounds, we can use the second resolvent formula and the continuous dependence of $T_{x,\tau}$ on its parameters to prove various continuity results about the resolvent.

\begin{lemma}
\label{lem:op-cont}
Suppose that $\bfs \in L^2(\mathbb C)$.  For any ${\tilde{r}}>2$, 
\vskip 0.1cm
(i) The mapping
	\begin{eqnarray*}
	\bbR^2 \times \bbR &\longrightarrow& \calB(L^{\tilde{r}}) \\
	(x,\tau) &\mapsto& \left( I - T_{x,\tau} \right)^{-1}
	\end{eqnarray*}
	is continuous.
\vskip 0.1cm
(ii) The mapping
	\begin{eqnarray*}
	L^2(\bbC) &\longrightarrow& C(\bbR^2 \times \bbR, \calB(L^{\tilde{r}}))\\
	\bfs &\mapsto &
				\left( 
							(x,\tau) \mapsto \left( I - T_{x,\tau} \right)^{-1} 
				\right)
	\end{eqnarray*}
	is continuous.
\end{lemma}

\begin{proof}
(i) This was already proved in the proof of Lemma \ref{lemma:Txtau}.
\vskip 0.1cm
(ii) Fix $(x,\tau)$ and write $T_{x,\tau}=T_{x,\tau}(\bfs)$ to emphasize the dependence on $\bfs$. From the second resolvent formula
\begin{eqnarray*}
\left(
		I - T_{x,\tau}(\bfs_1)
\right)^{-1} 
&
	- \left(
			I - T_{x,\tau}(\bfs_2)
	\right)^{-1}
	\\
&
	=	\left(I - T_{x,\tau}(\bfs_2)\right)^{-1}
			\left[ T_{x,\tau}(\bfs_1) - T_{x,\tau}(\bfs_2) \right]
			\left(I - T_{x,\tau}(\bfs_1)\right)^{-1}
\end{eqnarray*}
and the fact that
\[ \sup_{(x,\tau) \in \bbR^2 \times \bbR} \| T_{x,\tau}(\bfs_1) - T_{x,\tau}(\bfs_2) \|_{\calB(L^{\tilde{r}})} \leq C_{\tilde{r}} \| \bfs_1 - \bfs_2 \|_2 \]
we easily deduce that $\left\| \left( I - T_{x,\tau}(\bfs) \right)^{-1} \right\|_{\calB(L^{\tilde{r}})}$ is bounded uniformly in $(x,\tau)$ for $\bfs$ in a small metric ball in $L^2$ whose radius depends on the center but is uniform in $(x,\tau)$. For $\bfs_1$ and $\bfs_2$ in such a metric ball $B$ we may estimate
\[
\sup_{(x,\tau) \in \bbR^2 \times \bbR} \left\| \left(
		I - T_{x,\tau}(\bfs_1)
\right)^{-1} 
	- \left(
			I - T_{x,\tau}(\bfs_2)
	\right)^{-1}
\right\|_{\calB(L^{\tilde{r}})}
	\leq C(B) \| \bfs_1 - \bfs_2 \|_2.
\]
This gives the claimed continuity.
\end{proof}

We have already established in Lemma \eqref{lem:mu-diff}  that, for $\bfs \in \calX_{n,r}$, the derivatives $\dee_\tau^m D_x^\alpha \left( \mu(x,\dotarg,\tau) \right)$ exist in $L^{\tilde{r}}(\bbC)$, provided that $(3m+|\alpha|) \leq n$.  We are now ready to prove the continuity of several key maps in $\bfs$.

\begin{lemma}
\label{lem:mapcont}
Given $r\in (1,2)$
\vskip 0.1cm
(i) For any $n\geq 3m+|\alpha|$,  the maps
\begin{eqnarray*}
\calX_{n,r} &\longrightarrow& C(\bbR^2 \times \bbR, L^{\tilde r}(\bbC))\\
\bfs &\mapsto& \Bigl(\,  (x,\tau) \mapsto \dee_\tau^m D_x^\alpha(\mu(x,\cdot,\tau)-1) \, \Bigr)
\end{eqnarray*}
are continuous. 
\vskip 0.1cm
(ii) For any $n \geq 3m+|\alpha|+2$, the map
\begin{eqnarray*}
\calX_{n,r} &\mapsto& 			C(\bbR^2 \times \bbR)\\
\bfs			&\rightarrow&		\dee_\tau^m D_x^\alpha q(x,\tau)
\end{eqnarray*}
is continuous.
\end{lemma}

\begin{proof}
For any $n\geq 3m+|\alpha|$ we will show the map
\begin{eqnarray*}
\calX_{n,r} &\longrightarrow& C(\bbR^2 \times \bbR, L^{\tilde r}(\bbC))\\
\bfs &\mapsto& \Bigl( \, (x,\tau) 
	\mapsto \dee_\tau^m D^\alpha_x [T_{x,\tau}(\bfs)(1)] \Bigr)
\end{eqnarray*}
is continuous. The argument is similar to the proof of \cite[Lemma 5.4]{Music:2013}. From the computation
\begin{equation}
\label{TDerivs}
\dee_\tau^m D_x^\alpha T_{x,\tau}1 = P_k \left[ e^{itS} (-2ik_2)^{\alpha_1}(2ik_1)^{\alpha_2} [i(k^3+\kbar^3)]^m \bfs(\dotarg) \right]
\end{equation}
and \eqref{eq:P-1}, we see that it suffices to bound $\norm{ |\dotarg|^\ell (\bfs_1(\dotarg)-\bfs_2(\dotarg))}_{r}$ by $\norm{\bfs_1 - \bfs_2}_{\calX_{n,r}}$, where $\ell=3m+|\alpha|\leq n$. This is immediate from the definition.

Part (i) is the same as \cite[Lemma 5.5]{Music:2013} except here we keep track of the continuous dependence on the scattering data. We need to check that the expression
\eqref{mu.deriv} defines a continuous $C(\bbR^2 \times \bbR;L^{\tilde{r}}(\bbC))$-valued function of $\bfs$. By Lemma \ref{lem:op-cont}(ii), it suffices to show that 
$\bfs \mapsto P_k f$ (with $f$ given by \eqref{f}) has the same property. A typical term of $P_k f$ takes the form
$$ 
P_k \left[ \bfs(\dotarg) e^{itS} k^\beta (i((\dotarg)^3+(\overline{\dotarg})^3))^\ell 
	\left( \dee_\tau^\ell \dee_x^\beta \mu \right)(x,\dotarg,\tau) \right]
$$
where $|\beta|+\ell \leq m-1$. Assuming inductively that the derivatives $ \dee_\tau^\ell \dee_x^\beta \mu $ are continuous $C(\bbR^2 \times \bbR, L^{\tilde{r}}(\bbC))$-valued functions of $\bfs$, we can now use \eqref{eq:P-4} with $q=\tilde{r}$, $p=r$, and $s=2$ to obtain the required continuity. Thus it remains to prove that $\mu(x,\cdot,k)-1$ is a continuous $C(\bbR^2 \times \bbR;L^{\tilde{r}}(\bbC))$-valued function of $\bfs$. This is an immediate consequence of Lemma \ref{lem:op-cont}(ii) and equation \eqref{TDerivs}.

Part (ii) follows directly from part (i). The proof is the same as in Music \cite[Lemma 5.6]{Music:2013}. The result follows from equation \eqref{ISM}, part (i), and the fact that $\bfs\in \calX_{n+1,r}\subset L^1_{n}$.
\end{proof}

\section{Symmetries of the scattering transform}
\label{sec.symmetry}

In this section we prove that $q(x)$ is real if and only if $\bft(k)= \overline{\bft(-k)}$. This symmetry was proved in the negative energy case by Grinevich and Manakov in \cite{GM:1988} where they also sketch an argument for the symmetry in the zero-energy case. 

Here we will write this argument directly in the language of the zero-energy case. First, we note Lemma \ref{lemma:mu.q} which allows us to prove relevant formulas for smooth potentials and conclude that these formulas continue to hold for $q \in L^p_\rho(\bbR^2)$ by continuity. 
\refnote{Here we removed the erroneous proof}
Because the proof is somewhat lengthy, we defer it to \ref{app:symmetry}.

\begin{lemma}
If $q\in L^p_\rho(\mathbb R^2)$ for $p\in(1,2)$ and $\rho>2/p'$ is real and $k$ and $-k$ are not exceptional points, then $\bft(k)=\overline{\bft(-k)}$.
\label{lem:symmetry}
\end{lemma}

Now, assuming that the scattering data $\bft(k)$ has this symmetry, we prove that, if $q(x)$ is computed using the reconstruction formula \eqref{eq:q.recon}, then $q(x)=\overline{q(x)}$. For simplicity, we also assume that ${\bft}(k) = 0$ in a neighborhood of the origin. We follow the proof presented in \cite[Theorem 3.4]{Grinevich:2000}, just rewriting it in terms of the zero energy problem. 
\begin{lemma}
If $\bft \in C^\infty_c(\bbC)$ with ${\bft}(k)=0$ for $|k|<\epsilon$ for some $\epsilon>0$ and $\bft(k)=\overline{\bft(-k)}$ then $q(x)$ defined by equation \eqref{eq:q.recon} is real. Moreover, the solution $\mu(x,k)$ of the $\dbar$-problem \eqref{eq:dbar-k.static} satisfies
\[\dbar_x(\partial_x+ik)\mu(x,k)=q(x)\mu(x,k)\]
in distribution sense.
\label{lem:invsymmetry}
\end{lemma}
\begin{proof}
Consider the real differential form
\[
\omega=	\frac{\mu(x,k)\mu(x,-k)}{k} \, dk +
				\frac{\overline{\mu(x,k)\mu(x,-k)}}{\bar k} \, d\bar k.
\]
Using the symmetry $\bft(k)=\overline{\bft(-k)}$, and \eqref{eq:dbar-k.static}, it is not difficult to see that $\omega$ is a closed form. It now follows by Stokes' Theorem applied to the region $R^{-1} \leq |k| \leq R$, the large-$k$ asymptotic behavior of $\mu$ and $\bar\mu$ from \eqref{eq:largek-mu},  and the identities
\[\oint_\gamma \frac{dk}{k}=-\oint_\gamma \frac{d\bar k}{\bar k} = 2\pi i, \]
true for any simple closed contour $\gamma$, that
\begin{equation}
\label{zeroid}
\mu(x,0)^2 = \left[\overline{\mu(x,0)}\right]^2.
\end{equation}
Thus, $\mu(x,0)$ is either purely real or purely imaginary. 
Next consider the differential operators
\begin{eqnarray*}
P_1\psi &=&  -\dbar_x \left(\partial_x+ik\right) \psi+ q \psi \\
P_2\psi &=& -\partial_x   \left(\bar\partial_x -i\bar k\right) \psi + q\psi
\end{eqnarray*}
where $q$ is defined by the expansion \eqref{eq:largek-mu} to be $q=i\bar\partial_x a_1$. Let $\chi_1 = P_1\mu$ and $\chi_2=P_2\overline{\mu}$. We need to show $\chi_1=\chi_2 = 0$. From the expansion \eqref{eq:largek-mu}, we have
\begin{equation*}
\lim_{|k|\to \infty} \chi_1(x,k) = 0.
\end{equation*}
By \eqref{zeroid}, we have
\begin{equation}
\label{chizero}
\chi_1(x,0)^2 - \chi_2(x,0)^2 = 0,
\end{equation}
and $\chi_1$ and $\chi_2$ satisfy
\begin{eqnarray*}
\left(\dbar_k \chi_1\right)(x,k)  		&=& 	e_{-x}(k) {\bfs}(k)\chi_2(x,k)\\
\left(\partial_k  \chi_2\right)(x,k)  	&=&	e_{x}(k)  \overline{{\bfs}(k)} \chi_1(x,k).
\end{eqnarray*}
where, for each fixed $x$, 
\[\lim_{|k|\to\infty} \chi_2(x,k)=\lim_{|k|\to\infty} i\partial_x \bar a_1+q+O\left(|k|^{-1}\right) = i\partial_x \bar a_1(x)+q(x).\]
 It is not (yet) clear that $i\partial_x \bar a_1(x)+q(x)=0$ but we will prove this using the condition \eqref{chizero}.

To this end, consider  the one-form 
\[
\eta = \frac{\chi_1(x,k)\chi_1(x,-k)}{k}\, dk +\frac{\chi_2(x.k)\chi_2(x,-k)}{\bar k} \, d\kbar.
\]
A computation analogous to the one for $\omega$ together with \eqref{chizero} shows that $\eta$ is a closed form and that
\[
\chi_1(x,\infty)^2-\chi_2(x,\infty)^2 = \chi_1(x,0)^2-\chi_2(x,0)^2
\]
where $\chi_i(x,\infty)$ means $\lim_{|k| \to \infty} \chi_i(x,k)$ for $i=1,2$. 
By \eqref{chizero}, the right-hand side is zero, and also $\chi_1(x,\infty)=0$. We may conclude then that $\chi_2(x,\infty)=0$, thus $q=\overline{i\bar \partial_x a_1} = \overline{q}$ as desired.
\end{proof}
Combining the above result with Lemma \ref{lem:mapcont} we have:
\begin{corollary}
If $\bfs \in \calX_{2,r}(\bbC)$ for $r\in(1,2)$ and $\epsilon>0$ then $q(x)$ defined by equation \eqref{eq:q.recon} is real. Moreover, the solution $\mu(x,k)$ of the $\dbar$-problem \eqref{eq:dbar-k.static} satisfies
\[\dbar_x(\partial_x+ik)\mu(x,k)=q(x)\mu(x,k)\]
in the sense of distributions.
\label{cor:invsymmetry}
\end{corollary}

It is now simple to prove the interesting fact that all scattering transforms in $\calX_{n,r}$ come from critical or subcritical potentials.

\begin{proposition}
If $\bfs\in \calX_{2,r}$ then $q(x)=[\mathcal Q \bft](x)$ is critical or subcritical.
\label{prop:critorsubcrit}
\end{proposition}
\begin{proof}
 This result follows easily using an approximation argument. Define $q_\epsilon(x) = \mathcal Q(\bft_\epsilon)(x)$ where $\bft_\epsilon(k) = \chi_\epsilon(k)\bft(k)$ and $\chi_\epsilon$ is a smooth radial function which is $1$ for $|k|>\epsilon$ and $0$ for $|k|<\epsilon/2$. Nachman \cite[Theorem 4.1]{Nachman:1996} proves that the solutions $\mu_\epsilon(x,k)$ to equation \eqref{eq:dbar-k} satisfy $\inf_{x,k}|\mu_\epsilon(x,k)|>0$ and from \eqref{zeroid} $\mu_\epsilon(x,0)$ is either real or imaginary. By multiplying this by the appropriate constant and using Corollary \ref{cor:invsymmetry}, we have $c\mu_\epsilon(k,0)$ is a positive solution to the Schr\"odinger equation with potential $q_\epsilon$. By Lemma \ref{lem:mapcont} $q_\epsilon\in C(\mathbb R^2)$, so by \cite[Theorem 2.12]{Schro-Ops} $q_\epsilon$ is critical or subcritical. By Lemma \ref{lem:mapcont}, on any compact subset $\Omega\subset \mathbb R^2$ we have $\lim_{\epsilon\to 0}|q_\epsilon(x)-q(x)| = 0$.  By Definition \ref{def:murata}, a potential $q$ is either critical or subcritical if the associated quadratic form \eqref{q.ps} is positive. That is, for every $q_\epsilon$ and any $\psi \in C^\infty_c(\mathbb R^2)$
\[\int_{\mathbb R^2} |\partial_x \psi|^2+q_\epsilon|\psi|^2\,dm(x) \geq 0.\]
Taking limits preserves the non-negativity and the result follows.
\end{proof}

The following identities are a consequence of $\mu(x,k,\tau)$ solving 
\[\dbar_x(\partial_x+ik)\mu=q\mu\]
point-wise as a map $k\to L^{r}$ and the large-$k$ expansion \eqref{eq:largek-mu}.

\begin{corollary}
Assume $\bfs\in \calX_{n+1,r}$ and define $q = i\dbar_x a_1$.  Then, the following identities hold:
\begin{equation}
i\dbar_x a_{n} = -\dbar_x\partial_x a_{n-1}+(i\dbar_x a_1)a_{n-1},
\label{eq:mu-identities}
\end{equation}
which for $n=2$ simplifies to
\begin{equation}
i\dbar_x a_{2}=\dbar_x\left(-\partial_x a_1+i\frac{a_1^2}{2}\right).
\label{eq:dbar-a1}
\end{equation}
Additionally we have
\begin{equation}
i\partial_x a_{2} = \partial_x\left(-\partial_x a_1+i\frac{a_1^2}{2}\right).
\label{eq:d-a1}
\end{equation}
\label{cor:mu-identities}
\end{corollary}
\begin{proof}
Identity \eqref{eq:mu-identities} follows by plugging in the large-$k$ expansion \eqref{eq:largek-mu} into \eqref{mu.schro}. Equation \eqref{eq:dbar-a1} can be rearranged to 
\[\dbar_x\left(ia_2+\dee_x a_1-i\frac{a_1^2}{2}\right)=0.\]
 We then get identity \eqref{eq:d-a1} by applying Liouville's theorem to the analytic function $ia_2+\partial_x a_1-i\frac{a_1^2}{2}$ and noting that $a_1$, $a_2$, and $\partial_x a_1$ are bounded. 
\end{proof}

\section{Solving the Novikov-Veselov equation}
\label{sec.solving}

In this section, we show that $q(x,\tau)$ defined by \eqref{ISM} solves the NV equation at zero energy. We will assume $\bfs\in \calX_{5,r}$.  Note that, for any $\bfs$ of this form, the function $q(x,\tau)$ defined by \eqref{ISM} is real by Corollary \ref{cor:invsymmetry} and the fact that the map $\bft \mapsto e^{i\tau(k^3+\kbar^3)}\bft$ preserves the symmetry $\bft(k)=\overline{\bft(-k)}$ and the norms. For such transforms, we will prove that $q(x,\tau)$ does indeed solve the NV equation. It will then follow from Lemma \ref{lemma.scattering} that the inverse scattering method produces a global solution for critical or subcritical initial data $q_0$ with $q_0 \in W_{\rho}^{5,p}(\bbR^2)$ for some $p \in (1,2)$, and $\rho>1$.

 This section exploits ideas from Section 3 of the paper of Croke, Mueller, Music, Perry, Siltanen, and Stahel \cite{CMMPSS:2014}, where the Manakov triple representation of the NV equation \cite{Manakov:1976} is used to study the inverse scattering method. In order to show that $q(x,\tau)$ as defined by \eqref{ISM} in fact solves the NV equation, we need an equation of motion for $\mu(x,k,\tau)$.

We will show directly from the $\dbar_k$-equation for $\mu$ that $\mu(x,k,\tau)$ satisfies an explicit equation of 
motion. It then becomes a matter of careful computation to show that the time derivative of \eqref{ISM} satisfies the 
NV equation.

\begin{lemma}
If $\bfs\in \calX_{4,r}$ for $r\in(1,2)$ and $\epsilon>0$ then the function $\mu(x,k,\tau)$ defined by \eqref{eq:dbar-k} satisfies the equation
\begin{equation*}
(\dee_\tau-ik^3)\mu=\left[\dbar_x^3+(\partial_x+ik)^3 -3u(\partial_x+ik)-3\ubar \dbar_x \right]\mu(x,k)
\end{equation*}
where $u=i\partial_x a_1$.
\end{lemma}
\begin{proof}
We have the two equations \[\dbar_x(\partial_x+ik)\mu(x,k)=q(x)\mu(x,k)\]
and 
\[\dbar_k \mu(x,k,\tau) = {\bfs}(k,\tau)e_{-x}(k)\overline{\mu}\]
Expanding $\mu(x,k,\tau)$ up to order $k^3$, we have
\begin{equation}
\mu(x,k,\tau)= 1+\frac{a_1(x,\tau)}{k}+\frac{a_2(x,\tau)}{k^2}+ \frac{a_3(x,\tau)}{k^3}+ o(k^{-3}).
\label{eq:order3-mu}
\end{equation}
By commuting the following specially chosen differential operators through the $\bar\partial_k$ equation, we obtain the three identities
\[
	\dbar_k(\dee_\tau -ik^3)\mu = 
			{\bfs}(k,\tau)e_{-x}(k)\overline{(\dee_\tau -ik^3)\mu},
\]
\[
	\dbar_k[\dbar_x^3 +(\partial_x +ik)^3]\mu = 
			{\bfs}(k,\tau)e_{-x}(k)\overline{\left[\dbar_x^3 +(\partial_x +ik)^3\right]\mu}
,\]
and
\[
	\dbar_k[-3\ubar \dbar_x -3u(\partial_x +ik)]\mu = 
			{\bfs}(k,\tau)e_{-x}(k)\overline{\left[-3\ubar \dbar_x-3u(\partial_x +ik)\right]\mu}.
\]
We can combine these in such a way to get a formula for $\dee_\tau \mu$. Adding the three identities we conclude that
\[
	\Psi=\left[
		\dee_\tau -(\dbar_x^3+\partial_x^3+3ik\partial_x^2-3k^2\partial_x-3u(\partial_x+ik)-3\ubar\dbar_x )				\right]\mu 
\]
satisfies
\[\dbar_k \Psi = {\bfs}(k,\tau)e_{-x}(k) \overline{\Psi}.\]
We will prove that $\Psi=O(k^{-1})$ so that Theorem \ref{thm:vanish} implies that 
$\Psi\equiv 0$.
Plugging \eqref{eq:order3-mu} into the expression for $\Psi$, it is easy to see that we get no terms of order $k^2$ or higher. Collecting all terms of order $k^1$ we find
\[-3k\partial_x a_1-3iku = 0\]
which is zero by our choice of $u$.
Collecting terms of order $k^0$ gives us
\begin{equation}
3i\partial_x^2 a_1 -3\partial_x a_2-3iu a_1 = 3i \partial_x \left[\partial_x a_1+ia_2-\frac{i}{2}a_1^2\right].
\label{eq:orderk}
\end{equation}
which is zero by identity \eqref{eq:d-a1}. Thus, the order $k^0$ terms in equation \eqref{eq:orderk} are zero. This finishes the proof that $\Psi\equiv 0$.
\end{proof}

We are now ready to prove that $q(x,\tau)=i\bar \partial_x a_1(x,\tau)$ solves the Novikov-Veselov equation. This follows by looking at the terms of order $k^{-1}$ after plugging in the asymptotic expansion into the evolution equation for $\mu$ and using the identities from Corollary \ref{cor:mu-identities}.
\begin{corollary}
\label{cor:solveNV}
If $\bfs\in \calX_{5,r}$ then the function $q(x,\tau)$ satisfies
\[\dee_\tau q = \dbar_x^3q+\partial_x^3q-3\partial_x( uq)-3\dbar_x(\ubar q).\]
\end{corollary}
\begin{proof}
We expand $\Psi$ at order $k^{-1}$ to get
\[ \dee_\tau a_1 =\dbar_x^3a_1+\partial_x^3a_1+3i\partial_x^2a_2-3\partial_x a_3-3u\partial_x a_1-3iu a_2-3\ubar \dbar_x a_1.\]
Applying the operator $i\dbar_x$ to both sides we get
\[\dee_\tau q = \dbar_x^3q+\partial_x^3q-3\partial_x^2\dbar_x a_2-3i\partial_x\dbar_x a_3-3\dbar_x(u^2)+3\dbar_x(u a_2)-3\dbar_x(\ubar q).\]
Now we use equation \eqref{eq:mu-identities} with $n=3$ to get
\begin{eqnarray*}
\dee_\tau q &=& \dbar_x^3 q+\partial_x^3 q-3\partial_x^2\dbar_x a_2+3\partial_x^2\dbar_x a_2	\\
	&&	-3\partial_x(q a_2)-3 \dbar_x(u^2)
			+3\dbar_x(u a_2)-3\dbar_x(\ubar  q).\\
\dee_\tau q &=& \dbar_x^3 q+\partial_x^3 q-3 q\partial_x a_2-3 \dbar_x(u^2)+3u\dbar_x a_2-3\dbar_x(\ubar  q)
\end{eqnarray*}

The final step is to show 
\[-3 q\partial_x a_2+3u\dbar_x a_2-3 \dbar_x(u^2)= -3\partial_x(u q).\]
We use equations \eqref{eq:dbar-a1} and \eqref{eq:d-a1} to get
\begin{eqnarray*}
3i q\partial_x\left(-\partial_x a_1+i\frac{a_1^2}{2}\right)&& 
\hspace{-0.5cm}
-3iu\dbar_x\left(-\partial_x a_1+i\frac{a_1^2}{2}\right)- 3\dbar_x(u^2)\\
&=&-3 q\partial_x u +3u\partial_x q -6u\dbar_x u\\
&=&-3 q\partial_x u-3u\partial_x q\\
&=&-3\partial_x(u q)
\end{eqnarray*}
\end{proof}

The main result of our paper now follows trivially.

\noindent
\emph{Proof of Theorem \ref{thm:nv}}.
By Lemma \ref{lem:props-of-transform} we have $\bfs\in \calX_{5,r}$ when $q_0\in W^{5,p}_\rho(\mathbb R^2)$ for $p\in(1,2)$, $\rho>1$, $\epsilon>0$, and $r\in(\tilde p',\infty)$. By Corollary \ref{cor:solveNV}, $q(x,\tau)$ solves the Novikov-Veselov equation. By Lemma \ref{lem:mapcont} part (ii)
$q(\cdot, \tau) \in C(\bbR^2 \times \bbR)$ and hence $q(x,\tau) \rarr q(x,0)$, where  $q(x,0)$, the inverse scattering transform of $\bfs(k,0)$, is the initial datum.
\hfill $\Box$

\section{Conclusion}
\label{sec.open}

We have shown that we can solve the Novikov-Veselov equation via inverse scattering for a wide range of potentials. The most significant problem left in zero-energy inverse scattering is the handling of
supercritical initial data, for which the exceptional set need not be empty and the scattering transform may have non-integrable singularities of codimension one (see, for example \cite{MPS:2013}). In addition, Taimanov and Tsarev \cite[\S 4]{TT:2010} have constructed supercritical initial data for the NV equation corresponding to a solution that blows up in finite time. This solution currently lacks in interpretation in terms of inverse scattering, and such an interpretation would be of considerable interest. 

Another open problem concerns the decay of solutions with subcritical initial data. Lassas, Mueller, and Siltanen \cite{LMS:2007} prove that certain reconstructed critical potentials have the decay $|(\mathcal Q \bft)(x)|\leq C \langle x\rangle^{-2}$. We would like to know that this result, or a similar one, holds for subcritical potentials as well, but the singularity at $k=0$ in the scattering transform poses additional difficulties.

\appendix

\section{Symmetry Property of the Scattering Transform}
\label{app:symmetry}

%
%

\newcommand{\ptilde}{\widetilde{p}}

\newcommand{\scrB}{\mathscr{B}}

\newcommand{\calG}{\mathcal{G}}

\newcommand{\xibar}{\overline{\xi}}


\refnote{This appendix replaces the erroneous of Lemma 5.1 from the previous
version of the paper}
In this Appendix we give the proof of Lemma \ref{lem:symmetry}. We consider a slightly larger class of potentials but add the assumption  that $k$ and $-k$ are not exceptional points. Recall that, for the potentials considered in the body of the paper, there are no exceptional points.

Suppose that $q \in L^p(\R^2)$ for some $p \in (1,2)$ and that $q$ is real-valued. In what follows denote by $\ptilde$ the Sobolev conjugate to $p$, i.e., $1/\ptilde = 1/p -1/2$. We also set $p^* = (\ptilde)'$. 
Recall that the function 
$\mu(x,k)$ solves the integral equation
\begin{equation}
\label{mu}
\mu(x,k)  = 1 + g_k*\left( q(\dotarg) \mu(\dotarg,k) \right)(x)
\end{equation}
where $\mu(\dotarg,k) -1 \in W^{1,\ptilde}(\R^2)$. The convolution kernel $g_k$ is given by 
\begin{equation}
\label{gk}
g_k(x) = \frac{1}{\pi^2} \int \frac{e^{ix\cdot \xi}}{\xibar(\xi + 2ik)} \, d\xi
\end{equation}
where $x \cdot \xi = x_1 \xi_1 + x_2 \xi_2$ but, in the denominator of the integrand,  $\xi = \xi_1 + i \xi_2$ and $\xibar = \xi_1 - i\xi_2$. We note the symmetries 
\cite[Theorem 3.2]{Siltanen:1999}
\begin{equation}
\label{gk.sym}
g_k(x) = g_{-k}(-x), \quad \overline{g_k(x)} = e_k(x) g_k(x).
\end{equation}

We wish to prove that
\begin{equation}
\label{t.sym}
\overline{\bft(k)} = \bft(-k)
\end{equation}
whenever $q$ is a real-valued function and $k$ and $-k$ are not exceptional points. Note that the `Fourier' transform
$$ \widehat{q}(k) = \int e_k(x) q(x) \, dx $$
satisfies this symmetry trivially for real-valued functions $q$, so it will suffice to prove that
$$ \bft(-k) - \widehat{q}(-k) = \overline{\bft(k)} - \overline{\widehat{q}(k)}. $$

Let us define
$$ {\calG_k} f = g_k * f, \quad T_k f = {\calG_k}(qf) $$
The convolution operator ${\calG_k}$ is bounded from $L^p(\R^2)$ to $W^{1,\ptilde}(\R^2)$ \cite[Lemma 1.3]{Nachman:1996}.
The operator $T_k$ is a bounded operator from $W^{1,\ptilde}(\R^2)$ to itself since any
$f \in W^{1,\ptilde}(\R^2)$ belongs to $L^\infty(\R^2)$ (as $\ptilde>2$) so that $qf \in L^p(\R^2)$. Moreover,
$T_k$ is compact because $f \mapsto qf$ is a compact map from $W^{1,\ptilde}(\R^2)$ to $L^p(\R^2)$. 
The hypothesis on exceptional points means that the operators
$(I-T_k)^{-1}$ and $(I-T_{-k})^{-1}$ exist as bounded operators
on $W^{1,\ptilde}(\R^2)$.  Note that
$$ \mu -1  =  (I-T_k)^{-1} (T_k 1). $$
The expression $T_k 1 = g_k *q$ defines an element of $W^{1,\ptilde}(\R^2)$ because 
$q \in L^p(\R^2)$. 
We denote by $T_k'$ the Banach space adjoint of $T_k$, defined as a bounded 
operator on $W^{-1,{p^*}}(\R^2)$. It follows from the first symmetry of \eqref{gk.sym} that
$$ T_k' f = q {\calG_{-k}} f$$
so that $T_k'$ restricts to a bounded operator on $L^p(\R^2)$. 

We'll need the following Banach space version of a standard commutation formula in 
Hilbert space (see \cite{Deift78} for discussion and references).

%
%

\medskip

\noindent\textbf{Lemma A.1}
\emph{
Suppose that $X$ and $Y$ are Banach spaces, and $A:X \rarr Y$ and $B: Y \rarr X$ are 
bounded operators. Then
\begin{itemize}
\item[(i)]		If $(I-AB)^{-1}$ exists, then $(I-BA)^{-1}$ exists and 
				$$ (I-BA)^{-1} = I + B(I-AB)^{-1}A.$$
\item[(ii)]	If $(I-BA)^{-1}$ exists, then $(I-AB)^{-1}$ exists and
				$$ (I-AB)^{-1} = I + A(I-BA)^{-1}B.$$
\end{itemize}
}

\medskip

\begin{proof}
A straightforward computation.
\end{proof}

\medskip

As an immediate consequence, under the hypotheses of the lemma, the identity
\begin{equation}
\label{AB-BA-pre}
(I-BA)^{-1}B = B(I-AB)^{-1} 
\end{equation} 
holds in the bounded operators on $X$.

We can apply this to $A= q$, $B = \calG_{-k}$ (meaning the operator of multiplication by $q$), $X=W^{1,\ptilde}(\R^2)$, $Y=L^p(\R^2)$. It follows from Lemma A.1 and \eqref{AB-BA-pre} that, if $-k$ is not an exceptional point, $(I-q \calG_{-k})^{-1}$ exists as a bounded operator on $L^p(\R^2)$.
and the identity
\begin{equation}
\label{AB-BA}
\calG_{-k}(I- q \calG_{-k})^{-1} = (I-T_{-k})^{-1} \calG_{-k}
\end{equation}
holds in the sense of  bounded operators from $W^{1,\ptilde}(\R^2)$ to $L^p(\R^2)$. 

In what follows, we denote by 
$$\left\langle f,g \right\rangle= \int_{\R^2} f(x) g(x) \, dx,$$
which gives the usual dual pairing between $L^p(\R^2)$ and $L^{p'}(\R^2)$. If $A:X \rarr Y$ is a linear operator, $A':Y^* \rarr X^*$ is its Banach space adjoint, i.e., $\left\langle f , Ag \right\rangle = \left\langle A'f, g \right\rangle$. From \eqref{t} it is easy to see that
\begin{eqnarray}
\label{t-.rep}
\bft(-k)		 - \widehat{q}(-k)		
		&=&	\left\langle e_{-k} q , (I-T_{-k})^{-1} T_{-k} 1 \right\rangle\\[10pt]
\label{tbar.rep}
\overline{\bft(k)}-\overline{\widehat{q}(k)}	
						&=&	\left\langle e_{-k} q,  C(I-T_k)^{-1} T_k 1  \right\rangle\\
\nonumber
						&=&	\left\langle e_{-k} q,   (I - C T_k C)^{-1} C (T_k 1)  \right\rangle\\
\nonumber
						&=&	\left\langle (I-(CT_kC)')^{-1} e_{-k}q, C (T_k 1)  \right\rangle
\end{eqnarray}
where $Cf = \overline{f}$
and we used the fact that $C1 = 1$.  Using \eqref{gk.sym},  it is easy to see that 

\begin{equation}
\label{CTC1}
C T_k C f = e_k {\calG_k} \left(e_{-k} q f \right) 
\end{equation}
and a straightforward computation shows that
\begin{equation}
\label{CTC2}
(C T_k C)' f= q e_{-k} {\calG_{-k}} (e_k f). 
\end{equation}

We can now prove Lemma \ref{lem:symmetry}.

\medskip

\noindent\emph{Proof of Lemma \ref{lem:symmetry}}.
Compute
\begin{eqnarray*}
\overline{\bft(k)} - \overline{\widehat{q}(k)}
	&=	\left\langle (I - (C T_k C)')^{-1} e_{-k}q, C (T_k 1) \right\rangle		\\
	&=	\left\langle (I - qe_{-k} {\calG_{-k}} (e_k \dotarg ))^{-1} e_{-k} q, e_k \calG_k (e_{-k} q)  \right\rangle \\
	&=	\left\langle e_{-k} (I - q\calG_{-k})^{-1} q, e_k \calG_{k} (e_{-k} q) \right\rangle\\
	&=   	\left\langle \calG_{-k} (I-q \calG_{-k} )^{-1} q, e_{-k}q \right\rangle\\
	&=	\left\langle (I-T_{-k})^{-1} T_{-k} 1, e_{-k} q \right\rangle\\
	&=	\bft(-k)-\widehat{q}(-k)
\end{eqnarray*}
where in the first step we used \eqref{tbar.rep}, in the second step we used \eqref{CTC1} and \eqref{CTC2}, in the third step we used the identity
$$(I-e_{-k} A (e_{k} \dotarg))^{-1} = e_{-k} (I-A)^{-1} e_k, $$ 
in the fourth step we used 
$\left(\calG_{k}\right)' = \calG_{-k}$, in the fifth step we used \eqref{AB-BA}, and in the final step we used \eqref{t-.rep}. \hfill $\Box$

\end{document}